\documentclass[a4paper,11pt]{article}
\usepackage[english]{babel}
\usepackage{latexsym,amsfonts,amsmath,enumerate,graphics,enumerate,amsthm,tikz,hyperref}
\usepackage[affil-it]{authblk}

\let\epsilon=\varepsilon
\newtheorem{theorem}{Theorem}[section]
\newtheorem{lemma}[theorem]{Lemma}
\newtheorem{corollary}[theorem]{Corollary}

\newtheorem{proposition}[theorem]{Proposition}
\newtheorem{conjecture}[theorem]{Conjecture}

\theoremstyle{definition}
\newtheorem{definition}[theorem]{Definition}
\newtheorem{example}[theorem]{Example}
\newtheorem{remark}[theorem]{Remark}

\newcommand{\kommentaar}[1]{\begin{item}\item[]{\small\textsf{#1}}\end{itemize}}

\let\epsilon=\varepsilon
\let\phi=\varphi
\let\theta=\vartheta

\def\R{\mathbb{R}}
\def\C{\mathbb{C}}
\def\H{\mathbb{H}}

\def\F{\mathrm{Funk}}
\def\RF{\mathrm{RFunk}}

\newcommand{\comment}[1]{}
\newcommand{\norm}[1]{\left\Vert #1 \right\Vert}
\newcommand{\mat}[4]{\begin{pmatrix} #1 & #2 \\ #3 & #4 \\ \end{pmatrix} }

\newcommand{\intC}{C^\circ}

\makeatletter
\def\@maketitle{%
  \newpage
  \null
  \vskip 2em%
  \begin{center}%
  \let \footnote \thanks
    {\Large\bfseries \@title \par}%
    \vskip 1.5em%
    {\normalsize
      \lineskip .5em%
      \begin{tabular}[t]{c}%
        \@author
      \end{tabular}\par}%
    \vskip 1em%
    {\normalsize \@date}%
  \end{center}%
  \par
  \vskip 1.5em}
\makeatother

\begin{document}

\title{ \LARGE Denjoy-Wolff theorems for Hilbert's and Thompson's metric spaces}

\author{Bas Lemmens%
\thanks{Email: \texttt{B.Lemmens@kent.ac.uk}; Supported by EPSRC grant EP/J008508/1}}
\affil{School of Mathematics, Statistics \& Actuarial Science,
University  of Kent, Canterbury, Kent CT2 7NX, UK.}

\author{Brian Lins%
\thanks{Email: \texttt{blins@hsc.edu}}}
\affil{Box 131, Hampden-Sydney College, Hampden-Sydney, VA 23943, USA.}

\author{Roger Nussbaum%
\thanks{Email: \texttt{nussbaum@math.rutgers.edu} Partially supported by NSFDMS 1201328}}
\affil{Rutgers, The State University Of New Jersey, 110 Frelinghuysen Road,
Piscataway, NJ 08854-8019, USA.}

\author{Marten Wortel%
\thanks{Email: \texttt{marten.wortel@gmail.com}; Supported by EPSRC grant EP/J008508/1}}
\affil{Unit for BMI, North-West University, Private Bag X6001-209, Potchefstroom 2520, South Africa}

\date{}
\maketitle

\begin{abstract} 
We study the dynamics of fixed point free mappings on the interior of a normal, closed cone in a Banach space that are nonexpansive with respect to Hilbert's metric  or Thompson's metric. We establish several Denjoy-Wolff type theorems that confirm conjectures by Karlsson and Nussbaum for an important class of nonexpansive mappings. We also extend and put into a broader perspective results by Gaubert and Vigeral concerning the linear escape rate of such nonexpansive mappings. 
\end{abstract}

\section{Introduction}\label{sec1}
The classical Denjoy-Wolff theorem asserts that all orbits of a fixed point free holomorphic mapping $f\colon\mathbb{D}\to\mathbb{D}$ on the open unit disc $\mathbb{D}\subseteq \mathbb{C}$ converge to a unique point $\eta\in\partial \mathbb{D}$. Since the inception of the theorem by Denjoy \cite{De} and Wolff \cite{Wo1,Wo2} in the nineteen-twenties a variety of  extensions have been obtained; see  for example \cite{Ab,BKR1,BKR2,BKR3,KKR,ReS}. A detailed account of its history and an extensive  list of references can be found in the recent survey articles  \cite[Appendices G and H]{AIR}, \cite{Khb}, and  \cite{ReS1}. The problems  considered  in this paper originated in work by Beardon  \cite{Bea1,Bea2} and Karlsson \cite{Ka}, who extended the Denjoy-Wolff theorem to  fixed point free nonexpansive mappings on metric spaces that possess certain features of nonpositive curvature. Earlier studies of the Denjoy-Wolff theorem in the context of metric spaces can be found in \cite{GR1,GR2,Re}. 

A mapping $f\colon M\to M$ on a metric space $(M,\rho)$ is called 
{\em nonexpansive} if 
\[
\rho(f(x),f(y))\leq \rho(x,y)\mbox{\quad  for all }x,y\in M.
\]
Recall that each holomorphic self-mapping of the open unit disc $\mathbb{D}\subseteq \mathbb{C}$ is nonexpansive under the hyperbolic metric by the Schwarz-Pick lemma.

Particularly interesting examples of metric spaces that possess features of nonpositive curvature are Hilbert's metric spaces. Hilbert's metric spaces were introduced by Hilbert \cite{Hil} and play a role in the solution of his fourth problem; see \cite{Pap}. They are Finsler metric spaces that naturally generalize Klein's model of the real hyperbolic space.
To  define Hilbert's metric, let $\Sigma$ be a  convex set in a real vector space $X$ such that for each $x\neq y$ in $\Sigma$  the straight line, $\ell_{xy}$, through  $x$ and $y$ has the property that $\ell_{xy}\cap \Sigma$ is a (relatively) open, bounded subset of $\ell_{xy}$. On $\Sigma$, {\em Hilbert's metric} is given by   
\begin{equation}\label{eq:hil}
\delta(x,y) := \log\left(\frac{|x'-y|}{|x'-x|}\frac{|y'-x|}{|y'-y|}\right )\mbox{\quad for }x\neq y\in \Sigma,
\end{equation}
where $x',y'\in\partial \Sigma$ are the end-points of the segment $\ell_{xy}\cap\Sigma$ such that $x$ is between $x'$ and $y$, and $y$ is between $y'$ and $x$. 

For finite dimensional Hilbert's metric spaces, Karlsson and Nussbaum independently conjectured the following  generalization of the Denjoy-Wolff theorem; see \cite{Khb,NTMNA}. 
\begin{conjecture}\label{con:1} If $f\colon\Sigma\to\Sigma$ is a fixed point free mapping on a finite dimensional Hilbert's metric space $(\Sigma,\delta)$, then there exists a convex set $\Omega\subseteq \partial \Sigma$ such that for each $x\in\Sigma$ all accumulation points of the orbit $\mathcal{O}(x;f):=\{f^k(x)\colon k\geq 0\}$ lie in $\Omega$.  
\end{conjecture}
In fact, Nussbaum conjectured that the same assertion holds in infinite dimensions under additional compactness conditions on $f$; see \cite[Conjecture 4.21]{NTMNA}. 
Note that if $\Sigma$ is finite dimensional and its closure (in the usual topology) is strictly convex, then each convex subset of $\partial \Sigma$ reduces to a single point. Conjecture \ref{con:1} was shown to hold in case $\Sigma$ has a strictly convex closure by Beardon \cite{Bea2}, and for polytopes  by Lins \cite{Li1}. Further supporting evidence was obtained in \cite{AGLN,KN,LiN,NTMNA}. 

Important  examples of Hilbert's metric nonexpansive mappings arising in mathematical analysis come from nonlinear mappings on cones. Let $C$ be a closed cone with nonempty interior, $C^\circ$, in a normed space $X$. Suppose that there exists a strictly positive linear functional $\phi\in X^*$, i.e., $\phi(x)>0$ for all $x\in C\setminus\{0\}$, and let $\Sigma_\phi^\circ:=\{x\in C^\circ\colon \phi(x)=1\}$. If $f\colon C^\circ\to C^\circ$ preserves the partial ordering induced by $C$ and is homogeneous (of degree one), then the mapping $g\colon\Sigma_\phi^\circ\to\Sigma_\phi^\circ$ given by 
\begin{equation}\label{eq:g}
g(x)=\frac{f(x)}{\phi(f(x))}\mbox{\quad for }x\in\Sigma_\phi^\circ,
\end{equation}
is nonexpansive under $\delta$ on $\Sigma_\phi^\circ$; see \cite[Chapter 2]{LNBook}. Examples of such mappings $f\colon C^\circ\to C^\circ$ include reproduction-decimation  operators in the analysis of diffusions on fractals \cite{LiN,Metz}, dynamic programming operators in stochastic games (after a change of variables) \cite{RS}, and mappings arising in nonlinear Perron-Frobenius theory \cite{LNBook}. 

Among other results we shall establish the following Denjoy-Wolff type theorem for mappings $g$ given in (\ref{eq:g}). 
\begin{theorem}\label{thm:2} 
Let $C$ be a closed cone with nonempty interior in a finite dimensional vector space and 
$\phi\in X^*$ be a strictly positive functional. Suppose that $f\colon C^\circ\to C^\circ$ is an order-preserving homogeneous mapping with no  fixed point in $C^\circ$ and partial spectral radius $r_{C^\circ}(f)=1$. If there exists $x_0\in C^\circ$ such that either \begin{enumerate}[(a)]
\item $\mathcal{O}(x_0;f)$ has a compact closure in the norm topology, or, 
\item  $\lim_{k\to\infty}\|f^k(x_0)\|=\infty$, 
\end{enumerate}
then there exists a convex set  $\Omega\subseteq\partial \Sigma_\phi^\circ$ such that for each $x\in \Sigma_\phi^\circ$ the accumulation points of $\mathcal{O}(x;g)$, where $g$ is given by (\ref{eq:g}), lie in $\Omega$
\end{theorem} 
In fact, we shall prove a more general infinite dimensional version of this result,  see Theorems \ref{thm:7.1} and \ref{thm:7.2}. Unlike in finite dimensions, there need not exist a strictly positive linear functional $\varphi\in X^*$ if $C$ is infinite dimensional; see \cite[pages 48--57]{Kras}. In that case we shall consider scalings of order-preserving homogeneous mappings $f\colon C^\circ\to C^\circ$ by using  continuous homogeneous functions $q\colon C^\circ\to (0,\infty)$.   

We conjecture that  condition (a) or (b)  in Theorem \ref{thm:2} always holds. In other words, we believe that there does not exist an order-preserving homogeneous mapping $f\colon C^\circ\to C^\circ$, with $r_{C^\circ}(f)=1$, on the interior of a finite dimensional closed cone, and a point $x_0\in C^\circ$ such that $\mathcal{O}(x_0;f)$ is unbounded in the norm topology and $\mathcal{O}(x_0;f)$ has an  accumulation point in $\partial C$. At present we can only confirm this in case $C$ is a polyhedral cone; see Theorem \ref{thm:7.4}.

Thompson's metric \cite{Tho} is closely related to Hilbert's metric and is defined on the interior of a closed cone $C$ in a normed space $X$. For Thompson's metric we establish the following Denjoy-Wolff type theorem. 
\begin{theorem}\label{thm:3}
Let $C$ be a  closed cone with nonempty interior in a finite dimensional vector space $X$  and $f\colon C^\circ\to C^\circ$ be a fixed point free mapping which is nonexpansive under Thompson's metric. If $\mathcal{O}(x_0;f)$ has compact closure in the norm topology for some $x_0\in C^\circ$, then there exists a convex set $\Omega\subseteq \partial C$ such that for each $x\in C^\circ$ the accumulation points of $\mathcal{O}(x;f)$  lie in $\Omega$.
\end{theorem}
Again we shall establish a more general infinite dimensional version; see Theorem \ref{thm:3.1}, which confirms \cite[Conjecture 4.23]{NTMNA} under the additional condition that there exists a pre-compact  orbit in the norm topology. 

In Section 4 we will introduce  a spectral radius, $r_{C^\circ}(f)$, for order-preserving homogeneous mappings $f\colon C^\circ\to C^\circ$ and use it, not only to prove Theorems \ref{thm:7.1} and \ref{thm:7.2}, but also to extend some results concerning the linear escape rate in \cite{GV}; see Theorem \ref{thm:4.1.4}, Corollary \ref{cor:6.4} and Theorem \ref{thm:cwII}. 
In Section 5 we shall study  Funk and reverse-Funk horofunctions on the interiors of cones, and characterize them for symmetric cones; see Theorem \ref{thm:5.4}. We shall  use the Funk and reverse-Funk horofunctions to establish a  Wolff type theorem for order-preserving homogeneous mappings $f\colon C^\circ\to C^\circ$ (see Theorem \ref{thm:wolff}), which will play a role in the proof of Theorem \ref{thm:7.2}.  

We will start by collecting some basic concepts in the next section.

\section{Preliminaries}\label{sec2} 
A convex subset $C$ of a real vector space $X$ is called a {\em cone} if $C\cap (-C)=\{0\}$ and $\lambda C\subseteq C$ for all $\lambda\geq 0$. A cone $C$ induces a
partial ordering $\leq_C$ on $X$ by 
\[
x\leq_C y\mbox{\quad if }y-x\in C.
\]
Throughout this paper we shall assume that $C$ is a closed cone with nonempty interior, denoted $C^\circ$, in a real Banach space $(X,\|\cdot\|)$. We
shall often assume that $C$ is {\em normal}, i.e., there exists a constant $\kappa\geq 0$ such that 
$\|x\| \leq\kappa \|y\|$ whenever $0\leq_C x\leq_C y$. 

Given a   closed cone $C$  with nonempty interior in a Banach space $(X,\|\cdot\|)$ and  $u\in C^\circ$ the {\em order unit norm},
$\|\cdot\|_u$ on $X$ is defined by 
\[
\|x\|_u:=\inf\{\lambda \geq 0\colon -\lambda u\leq_C x\leq_C \lambda u\}.
\]  
Note that $C$ is a normal cone in $(X,\|\cdot\|_u)$ with normality constant $\kappa=1$. Moreover, the order interval $[-u,u] :=
\{x \in X: -u \leq x \leq u\}$ (which is the unit ball in $\norm{\cdot}_u$) is a neighborhood of $0$ in the original topology by \cite[Lemma 2.5]{AT}, and so
the topology generated by $\norm{\cdot}_u$ is coarser than the original topology. If $C$ is normal in $(X,\|\cdot\|)$, the order unit norm $\|\cdot\|_u$ is
equivalent with $\|\cdot\|$; see for example \cite[Theorems 2.8 and 2.63]{AT}.

A linear functional $\phi\colon X\to\mathbb{R}$ is said to be {\em positive} if $\phi( C)\subseteq [0,\infty)$. It is said to be {\em strictly positive} if
$\phi(x)>0$ for all $x\in C$ with $x\neq 0$. Note that each positive functional on $X$ is continuous with respect to $\|\cdot\|_u$, as $|\phi(x)|\leq \phi(u)$
for all $x\in X$ with $\|x\|_u\leq 1$. We denote the {\em dual cone} by $C^*$; so, 
\[
C^*:=\{\phi\in X^*\colon \phi (C )\subseteq [0,\infty)\} .
\]
Furthermore we define 
\[
\Sigma_u^*:=\{\phi\in C^*\colon \phi(u)=1\}.
\]
The following lemma collects some known facts concerning $\Sigma_u^*$. For the reader's convenience we include the proofs. 
\begin{lemma}\label{lem:2.1} Let $C$ be a closed cone  with nonempty interior in a Banach space $X$. For $u\in C^\circ$ the following
assertions hold:
\begin{enumerate}[(1)] 
\item $x\leq_C y$ if and only if $\phi(x)\leq \phi(y)$ for all $\phi\in \Sigma_u^*$. 
\item For $x\in X$,
\[
\|x\|_u =\sup\{|\phi(x)|\colon \phi\in\Sigma_u^*\}.
\]
\item The set $\Sigma_u^*$ is  norm bounded by $1/d(u, X \setminus C)$, where $d(u,X\setminus C) :=\inf\{\|u-v\|\colon v\in X\setminus C\}$, and  $\Sigma_u^*$ is weak* compact. Moreover, if $X$ is separable, then $\Sigma_u^*$ is a weak* sequentially compact and there exists a strictly positive
functional $\psi\in \Sigma_u^*$. 
\end{enumerate}
\end{lemma}
\begin{proof} To prove (1) note that if $x\not\leq_C y$, then $y-x\not\in C$. In that case there exist $\alpha\in\mathbb{R}$ and $\phi\in X^*$ such that
$\phi(y-x)<\alpha$ and $\phi(z)>\alpha$ for all $z\in C$ by the Hahn-Banach separation theorem. We can normalize $\phi$ such that
$\phi(u)=1$. Also note that $0=\phi(0)>\alpha$, so that $\phi(y)<\phi(x)$. As $\phi(\lambda z)>\alpha$ for all $\lambda>0$ and $z\in C$, we must have that
$\phi(z)\geq 0$ for all $z\in C$, and hence $\phi\in\Sigma_u^*$. The opposite implication is trivial. 

To prove (2) note that it follows from (1) that  for each $x\in X$, 
\begin{eqnarray*}
\|x\|_u & = & \inf\{\lambda\geq 0\colon -\lambda u\leq_C x\leq_C\lambda u\}\\
       & = & \inf\{\lambda\geq 0\colon -\lambda\leq \phi(x)\leq \lambda\mbox{ for all }\phi\in \Sigma_u^*\}\\
       & = & \sup\{|\phi(x)|\colon \phi\in\Sigma_u^*\}. 
\end{eqnarray*}

To prove (3) define for $x\in X$ the weak* continuous linear functional $\hat{x}\colon X^*\to\mathbb{R}$ by $\hat{x}(\phi)=\phi(x)$. So, 
\[
\Sigma_u^* =\left (\bigcap_{x\in C} \hat{x}^{-1}([0,\infty))\right )\cap \hat{u}^{-1}(\{1\}),
\]
which is a weak* closed subset of $X^*$. 

 Let $r := d(u, X \setminus C) > 0$. If $\norm{z} \leq r$ and $\phi \in \Sigma_u^*$, then $u \pm z \in C$ and so $\phi(u \pm z) \geq 0$ which yields
\[ -1 = - \phi(u) \leq \phi(z) \leq \phi(u) = 1. \]
Hence $|\phi(z)| \leq 1$, and so $\norm{\phi} \leq 1/r$. Therefore $\Sigma_u^*$ is contained in a multiple of the unit ball of $X^*$, which is weak* compact by
Banach-Alaoglu, and so $\Sigma_u^*$ is weak* compact.  

It is well-known that if $X$ is separable, then  bounded sets of $X^*$ are weak* metrizable. In that case $\Sigma_u^*$ is sequentially compact, and hence $\Sigma_u^*$ is separable. Let $(\phi_k)_k$ be a dense sequence in $\Sigma_u^*$, and define $\phi=\sum_{k\geq 1}
2^{-k}\phi_k\in\Sigma_u^*$. For $x\in C$ with $x\neq 0$, $\|x\|_u>0$, and hence by part (2) there exists $\sigma\in\Sigma_u^*$ with $\epsilon:=\sigma(x)>0$.
Consider the weak* neighborhood of $\sigma$,
\[
N_{\epsilon,\sigma}:=\{\phi\in\Sigma_u^*\colon |(\phi-\sigma)(x)|<\epsilon\}
= \{\phi\in\Sigma_u^*\colon |\phi(x)-\epsilon|<\epsilon\}.
\] 
As $(\phi_k)_k$ is dense in $\Sigma_u^*$, there exists $\phi_m\in N_{\epsilon,\sigma}$, and hence $\phi_m(x)>0$. This implies that 
\[
\psi(x) :=\sum_{k\geq 1} 2^{-k}\phi_k(x) \geq {2^{-m}} \phi_m(x)>0,
\]
which shows that $\psi$ is strictly positive. 
\end{proof}

The partial ordering $\leq_C$ induces an equivalence relation $\sim_C$ on $C$ by $x\sim_C y$ if there exist $0<\alpha\leq \beta $ such that 
\[
\alpha y\leq_C x\leq_C \beta y.
\] 
The equivalence classes are called {\em parts} of $C$. It is easy to verify that $C^\circ$ is a part of $C$.  Given $x \in X$ and $y \in C$,  we let  
\[
M(x/y):=M(x/y;C) =\inf\{\beta \in \R \colon x\leq_C \beta y\}.
\]

On $C$,  {\em Thompson's metric} is defined by 
\[
d_C(x,y) :=\log\left( \max\{M(x/y),M(y/x)\}\right)
\]
for all $x\sim_C y$, and $d_C(x,y)=\infty$ otherwise.  It was shown by Thompson \cite{Tho} that $d_C$ is a metric on each part of $C$, and its topology on $C^\circ$ coincides with the norm topology, if $C$ is a closed, normal cone in a Banach
space.

Furthermore,  on $C$, {\em Hilbert's (projective) metric} is defined by 
\[
\delta_C(x,y) := \log \left( M(x/y)M(y/x)\right)
\]
for $x\sim_C y$ and $d_H(x,y)=\infty$ otherwise. Note that $\delta_C(\mu x,\nu y) =\delta_C(x,y)$ for all $\mu,\nu>0$ and $x\sim_C y$. It is known that $\delta_C$ is a metric between pairs of rays in each part of $C$, if $C$ is closed; see \cite[Chapter 2]{LNBook}. Moreover, if there exists a strictly positive linear functional $\phi\in X^*$, then $\delta_C$ coincides with Hilbert's metric $\delta$ given in (\ref{eq:hil}) on  $\Sigma_\phi^\circ=\{x\in C^\circ\colon \phi(x)=1\}$. In finite dimensional spaces the set $\Sigma_\phi^\circ$ is bounded in the norm topology, but it may be unbounded in infinite dimensional normed spaces.  In this paper we shall work with $\delta_C$ rather than $\delta$, and consider it on subsets $\Sigma\subseteq C^\circ$ with the property that for each $y\in C^\circ$ there exists a unique $\lambda>0$ such that $\lambda y\in\Sigma$.

The following basic lemma  will be useful. 
\begin{lemma}\label{lem:2.2} 
Let $C$ be a closed cone with nonempty interior in a Banach space $X$  and $u\in C^\circ$. For each $x \in X$ and $y\in C^\circ$ we have that 
\[
M(x/y) = \sup_{\phi\in \Sigma_u^*}\frac{\phi(x)}{\phi(y)}
\]
and the supremum is attained. Moreover, $(x,y) \mapsto M(x/y)$ is a continuous map from $X \times \intC$ into $\R$.
\end{lemma}
\begin{proof}
By Lemma \ref{lem:2.1} (1),
\begin{align*}
 M(x/y) &= \inf \{ \beta \in \R\colon x \leq_C \beta y \} \\
 &= \inf \{ \beta \in \R\colon  \phi(x) \leq_C \beta \phi(y) \mbox{ for all }\phi \in \Sigma_u^* \} \\
 &= \sup \{ \phi(x)/\phi(y)\colon \phi \in \Sigma_u^* \}.
\end{align*}
The supremum is attained by weak* compactness of $\Sigma_u^*$.

For the second statement recall that  the $\norm{\cdot}_u$-topology is coarser than the $\norm{\cdot}$-topology on $X$, and so we may assume that $X$ is equipped with
$\norm{\cdot}_u$. By Lemma \ref{lem:2.1}, the map $x \mapsto \hat{x}$ is an isometric order isomorphism from $(X, \norm{\cdot}_u)$ into $(C(\Sigma_u^*),
\norm{\cdot}_\infty)$, where $\hat{x}(\phi) := \phi(x)$ and $\Sigma_u^*$ is equipped with the weak* topology. The continuity statement now follows from
the fact that the map is the composition of the continuous maps
\[ (x,y) \mapsto (\hat{x}, \hat{y}) \mapsto \left( \hat{x}, \frac{1}{\hat{y}} \right) \mapsto \frac{\hat{x}}{\hat{y}} \mapsto \sup_{\phi \in
\Sigma_u^*} \frac{\hat{x}}{\hat{y}} (\phi) = \sup_{\phi\in \Sigma_u^*}\frac{\phi(x)}{\phi(y)} = M(x/y). \]
\end{proof}

\section{A Denjoy-Wolff theorem for Thompson's metric}\label{sec3}
In this section we prove  a  Denjoy-Wolff type theorem for fixed point free Thompson's metric nonexpansive mappings $f\colon C^\circ\to C^\circ$, 
where one of the orbits of $f$ has a compact closure in the norm topology. 
As we are allowing infinite dimensional cones, some care must be taken	to ensure that all accumulation  points of the orbits of fixed point free nonexpansive mapping lie in $\partial C$. Indeed, in \cite{Ed} Edelstein gave an example of a fixed point free nonexpansive mapping $f\colon H\to H$ on a separable Hilbert space $H$ such that $\mathcal{O}(0;f)$ is unbounded in norm, but has $0$ as an accumulation point; see also \cite{Sta}. To exclude such situations we shall assume that the nonexpansive mapping satisfies the following property. 
\begin{definition}\label{def:fpp}
Let $C$ be a  normal, closed cone  with nonempty interior in a Banach space $(X,\|\cdot\|)$, and let $f\colon C^\circ\to C^\circ$ be a continuous mapping. We shall say that $f$ has the \emph{fixed point property on $C^\circ$ with respect to $d_C$} if for each bounded, convex, closed  subset $D$ of $(C^\circ,d_C)$ with $f(D)\subseteq D$ we have that $f$ has a fixed point in $D$. 
\end{definition}
Of course, if $X$ is finite dimensional, then every continuous mapping $f\colon C^\circ\to C^\circ$ has the fixed point property with respect to $d_C$ by the Brouwer fixed point theorem.  In infinite dimensional spaces sufficient conditions were obtained by Nussbaum in \cite[Theorem 3.10]{NTMNA} in terms of ``condensing functions and measures of noncompactness''. 

Let us now formulate the main result of this section.
\begin{theorem}\label{thm:3.1}
Let $C$ be a  normal closed cone  with nonempty interior in a Banach space $(X,\|\cdot\|)$, and let $f\colon C^\circ\to C^\circ$ be a fixed point free Thompson's metric nonexpansive mapping satisfying the fixed point property on $C^\circ$ with respect to $d_C$. If  $\mathcal{O}(x_0;f)$ has compact closure in the norm topology for some $x_0\in C^\circ$, then there exists a convex set $\Omega\subseteq \partial C$ such that $\omega(x;f)\subseteq \Omega$ for all $x\in C^\circ$.  
\end{theorem}
This result confirms \cite[Conjecture 4.23]{NTMNA} by Nussbaum in case the mapping has a pre-compact orbit. Also note that Theorem \ref{thm:3.1} implies Theorem \ref{thm:3}, as each nonexpansive mapping has the fixed point property on $C^\circ$ with respect to $d_C$ when $C$ is finite dimensional.

The following proposition plays a central role in the proof. 
\begin{proposition}\label{prop:3.2} 
If $(x_k)_k$ is a  sequence in the interior of a  closed  cone $C$ in a Banach space $X$ such that $\{x_k\colon k=0,1,2,\ldots\}$ has
compact closure in the norm topology on $X$,  
\begin{equation}\label{eq:3.1} 
d_C(x_{m+k},x_{n+k})\leq d_C(x_m,x_n)\mbox{ \quad for all }k,m,n\geq 0,
\end{equation}
and 
\begin{equation}\label{eq:3.2} 
\lim_{k\to\infty} d_C(x_k,x_0)=\infty,
\end{equation}
then there exist a subsequence $(x_{k_i})_i$ and  $\eta\in\partial C$ such that 
\[
\lim_{i\to\infty} x_{k_i}= \eta,
\] 
and 
\begin{equation}\label{eq:3.3} 
d_C(x_m,x_0)<d_C(x_{k_i},x_0)\mbox{\quad for all }m<k_i.
\end{equation}
Moreover, there exist $\phi, \sigma_{jm}\in \Sigma_u^*$ for $j,m\geq 1$ with $\sigma_{jm}(\eta)=0$ such that  
\begin{equation}\label{eq:3.4} 
\frac{\phi(x_{k_j+m})}{\phi(x_0)}\leq \frac{\sigma_{jm}(x_{k_j})}{\sigma_{jm}(x_0)}
\end{equation}
for all $j,m\geq 1$. 
\end{proposition}
\begin{proof}
 Take $u\in C^\circ$ fixed, and let $R := 1/d(u, X \setminus C)$ so that $\Sigma_u^*$ is bounded by $R$ by Lemma~\ref{lem:2.1}(3). 
Let $Y$ be the closed linear span of 
 $\{x_k\colon k=0,1,2,\ldots\}\cup\{u\}$ and write $K=C\cap Y$.  So, $Y$ is separable and $K$ is a closed cone in $Y$ with $u$ in its interior. Note that 
 \[
 M(x/y;C)= M(x/y;K)\mbox{\quad for all }x,y\in K^\circ,
 \]
 and hence $d_K(x,y)=d_C(x,y)$ on $K^\circ$. As $Y\cap C^\circ$ is nonempty, $Y$ a majorizing subspace of $X$, meaning that for each $x\in X$ there exists $y\in
Y$ such that $x\leq_C y$. So by Kantorovich's theorem \cite[Theorem 1.30]{AT}, we know that each positive linear functional on $Y$ can be extended as a positive
functional to all of $X$. Thus, we may as well assume from the outset that $X$ is separable. 
 
 From Lemma \ref{lem:2.1}(3) we know that $\Sigma_u^*$ is weak* sequentially compact. 
 By (\ref{eq:3.2}) we can always find a subsequence $(x_{k_i})_i$ such that 
 (\ref{eq:3.3}) holds. Furthermore, as $\{x_k\colon k=0,1,2,\ldots\}$ has compact closure in the norm topology, we can take a further subsequence and assume
that $(x_{k_i})_i$ converges to $\eta\in\partial C$. (Note that $\eta$ cannot lie in $C^\circ$ as otherwise
$d_C(\eta,x_0)<\infty$, which violates (\ref{eq:3.2}).)
 
From Lemma \ref{lem:2.2} we know that for each $i,j,m\geq 1$ with $m\leq k_i$ there exist  $\phi_i,\psi_{im},\sigma_{ijm}\in\Sigma_u^*$ such that 
 \begin{gather}
 M(x_0/x_{k_i}) = \frac{\phi_i(x_0)}{\phi_i(x_{k_i})},\notag\\
 M(x_0/x_{k_i-m}) = \frac{\psi_{im}(x_0)}{\psi_{im}(x_{k_i-m})}, \label{eq:3.5}\\
 M(x_{k_j}/x_{k_i-m}) = \frac{\sigma_{ijm}(x_{k_j})}{\sigma_{ijm}(x_{k_i-m})}.\notag
 \end{gather}
 We claim that for all $i$ sufficiently large the Thompson's metric distance is equal to the logarithm of the $M$-functions in (\ref{eq:3.5}). We shall prove
this for $M(x_{k_j}/x_{k_i-m})$. The arguments for the other  functions are similar and are left to the reader. 

First note that 
\[
d_C(x_{k_i},x_0)-d_C(x_{k_j +m},x_0)\leq d_C(x_{k_j+m},x_{k_i})\leq d_C(x_{k_j},x_{k_i-m}).
\] 
As the left hand side goes to $\infty$ as $i\to\infty$, we find that $d_C(x_{k_j},x_{k_i-m})\to\infty$ as $i\to \infty$. For $i\geq 1$ let
$\hat{\sigma}_{ijm}\in\Sigma_u^*$ be such that 
\[
M(x_{k_i-m}/x_{k_j}) = \frac{\hat{\sigma}_{ijm}(x_{k_i-m})}{\hat{\sigma}_{ijm}(x_{k_j})}.
\] 
By weak* compactness of $\Sigma_u^*$ the sequence $(\hat{\sigma}_{ijm}(x_{k_j}))_i$ is bounded  from below by a positive real, and hence 
\[
M(x_{k_i-m}/x_{k_j}) = \frac{\hat{\sigma}_{ijm}(x_{k_i-m})}{\hat{\sigma}_{ijm}(x_{k_j})}
\leq  \frac{ R \|x_{k_i-m}\|}{\hat{\sigma}_{ijm}(x_{k_j})}
\]
is bounded from above by a positive real, since $(\|x_{k_i-m}\|)_i$ is bounded. Thus, for all $i$ sufficiently large we have that 
\[
d_C(x_{k_j},x_{k_i-m}) =\log \frac{\sigma_{ijm}(x_{k_j})}{\sigma_{ijm}(x_{k_i-m})}.
\]

From now on we will assume that $i$ is so large that the Thompson's metric distance is given by the logarithm of the $M$-functions in (\ref{eq:3.5}). 

By (\ref{eq:3.3}), 
\[
\log\left( \frac{\psi_{im}(x_0)}{\psi_{im}(x_{k_i-m})}\right) = d_C(x_0,x_{k_i-m})\leq d_C(x_0,x_{k_i}) = \log\left(
 \frac{\phi_i(x_0)}{\phi_i(x_{k_i})}\right),
\]
so that 
\begin{equation}\label{eq:3.6} 
\frac{\phi_i(x_{k_i})}{\psi_{im}(x_{k_i-m})} \leq  \frac{\phi_i(x_0)}{\psi_{im}(x_0)}.
\end{equation}
Note also that by definition of $\psi_{im}\in\Sigma_u^*$ we have that 
\[
\frac{\sigma_{ijm}(x_0)}{\sigma_{ijm}(x_{k_i-m})}\leq \frac{\psi_{im}(x_0)}{\psi_{im}(x_{k_i-m})},
\] 
so that 
\begin{equation}\label{eq:3.7} 
\frac{\psi_{im}(x_{k_i-m})}{\sigma_{ijm}(x_{k_i-m})}\leq \frac{\psi_{im}(x_0)}{\sigma_{ijm}(x_0)}.
\end{equation}

Now using equations (\ref{eq:3.1}), (\ref{eq:3.6}) and (\ref{eq:3.7}) we get that 
\begin{eqnarray*}
\frac{\phi_i(x_{k_j+m})}{\phi_i(x_0)} & = & \frac{\phi_i(x_{k_j+m})}{\phi_i(x_{k_i})}\frac{\phi_i(x_{k_i})}{\phi_i(x_0)}\\
 &\leq & e^{d_C(x_{k_j+m},x_{k_i})}\frac{\phi_i(x_{k_i})}{\phi_i(x_0)}\\
 &\leq & e^{d_C(x_{k_j},x_{k_i-m})}\frac{\phi_i(x_{k_i})}{\phi_i(x_0)}\\
 & = & \frac{\sigma_{ijm}(x_{k_j})}{\sigma_{ijm}(x_{k_i-m})} \frac{\phi_i(x_{k_i})}{\phi_i(x_0)}\\
 & = & \frac{\sigma_{ijm}(x_{k_j})}{\phi_i(x_0)} \frac{\phi_i(x_{k_i})}{\psi_{im}(x_{k_i-m})}   \frac{\psi_{im}(x_{k_i-m})} {\sigma_{ijm}(x_{k_i-m})}\\
& \leq & \frac{\sigma_{ijm}(x_{k_j})}{\phi_i(x_0)} \frac{\phi_i(x_0)}{\psi_{im}(x_0)}   \frac{\psi_{im}(x_0)} {\sigma_{ijm}(x_0)}\\
 & = & \frac{\sigma_{ijm}(x_{k_j})}{\sigma_{ijm}(x_0)}.
\end{eqnarray*}
As $\Sigma_u^*$ is sequentially weak* compact, we can pass to a subsequence twice and assume that $\phi_i\to \phi\in\Sigma_u^*$ and $\sigma_{ijm}\to\sigma_{jm}\in\Sigma_u^*$
in the weak* topology as $i\to\infty$, which proves (\ref{eq:3.4}). 

It remains to show that $\sigma_{jm}(\eta)=0$ for all $j,m\geq 1$. 
As 
\[
d_C(x_{k_j},x_{k_i -m}) = \log \left( \sigma_{ijm}(x_{k_j})/\sigma_{ijm}(x_{k_i-m})\right)\to\infty
\] 
as $i\to\infty$, we know that $\sigma_{ijm}(x_{k_i-m})\to 0$ as $i\to\infty$. Moreover, 
\[
\log  \frac{\sigma_{ijm}(x_{k_i})}{\sigma_{ijm}(x_{k_i-m})}\leq d_C(x_{k_i},x_{k_i-m})\leq 
d_C(x_m,x_0),
\]
which implies that $\sigma_{ijm}(x_{k_i})\to 0$ as $i\to\infty$. As
\[
\lim_{i\to\infty} |\sigma_{ijm}(x_{k_i}-\eta)| \leq \lim_{i\to\infty} R \|x_{k_i}-\eta\| = 0
\]
and $\sigma_{ijm}(x_{k_i}-\eta)\to -\sigma_{jm}(\eta)$ as $i\to \infty$, we see that $\sigma_{jm}(\eta) =0$.
\end{proof}
Before we proceed we mention a useful result due to Ca\l ka \cite{Ca}. Recall that a metric space $(M,\rho)$ is said to be \emph{finitely totally bounded} if for each bounded set $S\subseteq M$ and each $\epsilon>0$, the set $S$ can be covered with finitely many balls of radius $\epsilon$. 
 \begin{theorem}[Ca\l ka]
If $f\colon M\to M$ is a nonexpansive mapping on a finitely totally  bounded metric space 
$(M,\rho)$ and there exists $x_0\in M$ such that $\mathcal{O}(x_0;f)$ has a bounded subsequence, then $\mathcal{O}(x;f)$ is bounded for every $x\in M$. 
\end{theorem}
Using Ca\l ka's theorem and Proposition \ref{prop:3.2} we now derive  the following consequence. 
\begin{corollary}\label{cor:3.3} 
Let $C$ be a  normal closed  cone with nonempty interior in a Banach space $(X,\|\cdot\|)$  and let $f\colon C^\circ\to C^\circ$ be a fixed point free
Thompson's metric nonexpansive mapping satisfying the fixed point property on $C^\circ$ with respect to $d_C$. If $x_0\in C^\circ$ is such that  $\mathcal{O}(x_0;f)$ has compact closure in the norm topology, then there exists $\phi\in C^*\setminus\{0\}$ such that 
$\omega(x_0;f)\subseteq \mathrm{ker}(\phi)\cap C$. 
\end{corollary}
\begin{proof}
Take $u\in C^\circ$ fixed, and let $R := 1/d(u, X \setminus C)$. Recall that $\Sigma_u^*$ is bounded by $R$ by Lemma \ref{lem:2.1}(3). 

We will first prove that 
\[
\lim_{k\to\infty}d_C(f^k(x_0),x_0)=\infty.
\]
Suppose, by way of contradiction,  that there exists $r>0$ and a subsequence $(f^{k_i}(x_0))_i$ such that $d_C(f^{k_i}(x_0),x_0)\leq r$ for all $i$. Define $M$ to be the norm closure of $\mathcal{O}(x_0;f)$, which is compact in the norm topology by assumption. As the topology of $d_C$ coincides with the norm topology on $C^\circ$, we know that for each closed ball, $B_\delta(x_0):=\{y\in C^\circ\colon d_C(x_0,y)\leq \delta\}$, we have that $M\cap B_\delta(x_0)$ is compact with respect to $d_C$. So, for each $\epsilon>0$, the set  $\mathcal{O}(x_0;f)\cap B_\delta(x_0)$ can be covered with finitely many balls of radius $\epsilon$. This shows that $(\mathcal{O}(x_0;f),d_C)$ is finitely totally bounded. Using Ca\l ka's theorem we see that $\mathcal{O}(x_0;f)$ is bounded with respect to $d_C$. This implies that $\omega(x_0;f)\subseteq C^\circ$ is a nonempty and bounded with respect to $d_C$. As $f$ has the fixed point property on $C^\circ$ with respect to $d_C$, we can apply \cite[Theorem 3.11]{NTMNA} to conclude that the mapping $f$ has a fixed point in $C^\circ$, which contradicts our assumption.

For $k\geq 1$ let $x_k:=f^k(x_0)$. So, $(x_k)_k$ satisfies the assumptions of Proposition \ref{prop:3.2} and we find $\phi,\sigma_{jm}\in\Sigma_u^*$ such that 
\begin{equation}\label{eq:3.8} 
\frac{\phi(x_{k_j+m})}{\phi(x_0)}\leq \frac{\sigma_{jm}(x_{k_j})}{\sigma_{jm}(x_0)}
\end{equation} 
and $\sigma_{jm}(\eta)=0$ where $x_{k_i}\to\eta\in\partial C$ as $i\to\infty$. 

By weak* compactness of $\Sigma_u^*$, the sequence $(\sigma_{jm}(x_0))_j$ is bounded from below, and 
\[
\lim_{j\to\infty} |\sigma_{jm}(x_{k_j})| = \lim_{j\to\infty} |\sigma_{jm}(x_{k_j}-\eta)| 
\leq \lim_{j\to\infty}  R \|x_{k_j}-\eta\| = 0
\]
by Lemma \ref{lem:2.1}(2). This implies that the right hand side of (\ref{eq:3.8}) converges to $0$ uniformly in $m$, and hence 
\begin{equation}\label{eq:3.9} 
\lim_{j\to\infty} \phi(x_{k_j+m}) =0
\end{equation}
uniformly in $m$. 

Now if $\xi\in\omega(x_0;f)$, then there exists a subsequence $(x_{k_{j_n}+m_n})_n$ with $x_{k_{j_n}+m_n}\to \xi$ and $k_{j_n}\to\infty$ as $n\to\infty$.  From (\ref{eq:3.9}) it follows that $\phi(\xi)=0$ and hence $\xi$ is in the kernel, $\mathrm{ker}(\phi)$,  of the positive functional $\phi$; so,  $\omega(x_0,f)\subseteq \mathrm{ker}(\phi)\cap C$. 
\end{proof}
We can now prove Theorem \ref{thm:3.1} as follows. 
\begin{proof}[Proof of Theorem \ref{thm:3.1}]
We know from Corollary \ref{cor:3.3} that $\omega(x_0;f)$ is contained in $\partial C$. It remains to show that the convex hull of $\cup_{x\in C^\circ} \omega(x;f)$, denoted $\Omega$, is contained in $\partial C$. The argument is similar to the one given in \cite[Theorem 5.3]{NTMNA} and relies on the fact that the closure of $\mathcal{O}(x_0;f)$ is compact.. 

Let $z\in C^\circ$ and $\zeta\in\omega(z;f)$. So, there exists a subsequence $(f^{k_i}(z))_i$ converging to $\zeta$ in the norm topology. As $\mathcal{O}(x_0;f)$ has a compact closure, we may assume, after possibly taking a further subsequence, that $f^{k_i}(x_0)$ converges to some $\xi\in\omega(x_0;f)$. 
Obviously, $d_C(f^{k_i}(x_0),f^{k_i}(z))\leq d_C(x_0,z)$ for all $i$, and hence $\xi\sim_C\zeta$ by \cite[Lemma 5.2]{NTMNA}. 

Now let $\eta\in\Omega$. Then there exist $z_1,\ldots,z_n\in C^\circ$, $0<\lambda_1,\ldots,\lambda_n<1$ with $\sum_{i=1}^n \lambda_i=1$, and $\zeta_i\in\omega(z_i;f)$ for $i=1,\ldots,n$ such that $\eta=\sum_{i=1}^n \lambda_i\zeta_i$. 
For each $i=1,\ldots,n$ there exists $\xi_i\in\omega(x_0;f)$ with $\xi_i\sim_C \zeta_i$. 
Clearly $\nu:=\sum_{i=1}^n\lambda_i \xi_i$ is in the convex hull of $\omega(x_0;f)$ and $\nu\sim_C \eta$.

Now suppose that there exists $\chi\in\Omega\cap C^\circ$. By the previous observation we know that there exists $\nu$ in the convex hull of $\omega(x_0;f)$ with $\nu\sim_C \chi$. But this implies that $\nu\in C^\circ$, which contradicts the fact that $
\omega(x_0;f)$ is contained in $\partial C$.
\end{proof} 

\section{The cone spectral radius}
In the remainder of this paper we will discuss Denjoy-Wolff type theorems for Hilbert's metric nonexpansive mappings that come  from scaling  order-preserving homogenous mappings $f\colon C^\circ\to C^\circ$. More precisely, we consider mappings $g\colon \Sigma^\circ\to\Sigma^\circ$ of the form 
\[
g(x) := \frac{f(x)}{q(f(x))}\mbox{\quad for }x\in \Sigma^\circ_q,
\]
where $f\colon C^\circ\to C^\circ$ is an order-preserving homogeneous mapping on the interior of a normal closed cone in a Banach space $(X,\|\cdot\|)$, $q\colon C^\circ \to (0,\infty)$ is a norm continuous homogeneous  function and
\[
\Sigma^\circ_q:= \{x\in C^\circ\colon q(x) =1\}.
\] 
Typical examples of functions $q$ include strictly positive functionals $\varphi$ in $X^*$, $q(\cdot)=\|\cdot \|$, and $q(\cdot )=\|\cdot\|_u$ where $u\in C^\circ$ is fixed. 

To analyze the dynamics of such mappings, we need to introduce a spectral radius for 
order-preserving homogenous mappings $f\colon C^\circ\to C^\circ$. There exist various definitions for the spectral radius for continuous, order-preserving, homogeneous mappings $f\colon C\to C$ if $f$ is defined on the whole of the closed  cone $C$; see \cite{MN}. 
In general, however, $f\colon C^\circ\to C^\circ$ may fail to have a continuous, order-preserving, homogeneous extension to the whole of $C$; see \cite{BNS}.  So, some additional analysis is needed.

\subsection{Approximate eigenvectors}
As the definition of, and the results concerning, the cone spectral radius for mappings $f\colon C^\circ\to C^\circ$ are of some independent interest, we shall work in a slightly more general setting. In fact, we shall consider homogeneous mappings that are defined on  a subset of a normal closed  cone $C\subseteq X$ and that are order-preserving with respect to a, possibly different, normal closed cone $K\subseteq X$ with $C\subseteq K$. 

So, throughout this section we assume that $C\subseteq K$ are normal closed  cones in a Banach space $(X,\|\cdot\|)$. For $u\in C$ with $\|u\|=1$, we denote the part of $u$ (with respect to $K$)  by 
\[
K_u:=\{x\in K\colon \alpha x\leq_K u\leq_K\beta x\mbox{ for some }0<\alpha\leq \beta\}.
\]
We shall consider homogeneous mappings $f\colon C\cap K_u\to C\cap K_u$ that are order-preserving with respect to $K$, so $f(x)\leq_K f(y)$ whenever $x,y\in C\cap K_u$ and $x\leq_K y$. For the applications in this paper we shall eventually assume that $u\in C^\circ$ and $K=C$, in which case $K_u=C^\circ$. The reader may wish to make this simplifying assumption. 

\begin{definition}\label{def:4.1} Suppose that $u\in C$, with $\|u\|=1$, and $f\colon C\cap K_u\to C\cap K_u$ is homogeneous and order-preserving with respect to $K$. We say that $f$ is {\em $u$-bounded} if there exists $M>0$ such that 
\[
f(x)\leq_K M\|x\|u\mbox{\quad for all }x\in C\cap K_u.
\]
\end{definition}
Note that if $u\in C^\circ$, with $\|u\|=1$, and $K=C$, then any homogeneous order-preserving mapping $f\colon C^\circ\to C^\circ$ is $u$-bounded. Indeed, as $C$ is a closed normal cone and $u\in C^\circ$, the order-unit norm $\|\cdot\|_u$ is equivalent to $\|\cdot\|$. So there exists a constant $M_1>0$ such that $\|x\|_u\leq M_1$ for all $x\in C$ with $\|x\|\leq 1$. This implies that $x\leq_C M_1 u$ for all $x\in C$ with $\|x\|\leq 1$, and hence $f(x)\leq_C M_1f(u)$. As $u\in C^\circ$, there exists $M_2>0$ such that $f(u)\leq_C M_2 u$, so that 
\[
f(x)\leq_C M_1M_2 u\mbox{\quad for all $x\in C$ with $\|x\|\leq 1$}. 
\]

Given a homogeneous mapping $f\colon C\cap K_u\to C\cap K_u$ which is order-preserving with respect to $K$, we define for $k\geq 1$, 
\[
\|f^k\|_{C\cap K_u}:=\sup\{\|f^k(x)\|\colon x\in C\cap K_u\mbox{ and } \|x\|\leq 1\}.
\]

\begin{lemma}\label{lem:4.1.1}If $f\colon C\cap K_u\to C\cap K_u$ is a homogeneous mapping which is order-preserving with respect to $K$ and there exists an integer $m\geq 1$ such that $f^m$ is $u$-bounded, then $\|f^k\|_{C\cap K_u}<\infty$ for all $k\geq m$, and 
\[
\lim_{k\to \infty}\|f^k\|^{1/k}_{C\cap K_u}=\inf_{k\geq m}\|f^k\|^{1/k}_{C\cap K_u} = \lim_{k\to\infty} \|f^k(u)\|^{1/k}.
\] 
\end{lemma}
\begin{proof}
We first show that $f^k$ extends continuously to $0$ for all $k\geq m$. Note that if $k\geq m$ and $\|x_n\|\to 0$, then 
\[
f^k(x_n)  =  f^{k-m}(f^m(x_n)) \leq_K  M\|x_n\|f^{k-m}(u) \leq_K M\|x_n\|\beta_k u,
\]
for some $\beta_k>0$, as $f^m$ is $u$-bounded and $f^{k-m}(u)\in C\cap K_u$.  Thus, for each $k\geq m$, we have that $f^k(x_n)\to 0$ if  $\|x_n\|\to 0$. So, if we define $f^k(0):=0$, we obtain a continuous extension of $f^k$ to $0$. 

By using the homogeneity of $f^k$ it  is easy to show that  $\|f^k\|_{C\cap K_u}<\infty$ for all $k\geq m$. Using sub-additivity we now show that 
\[
\lim_{k\to \infty}\|f^k\|^{1/k}_{C\cap K_u}=\inf_{k\geq m}\|f^k\|^{1/k}_{C\cap K_u}.
\]
Let $a_n:= \log \|f^n\|_{C\cap K_u}$ for all $n\geq 1$. We know that $a_n<\infty$ for all $n\geq m$ and clearly $a_{p+q}\leq a_p+a_q$ for all $p,q\geq m$. Let 
\[L:= \inf_{n\geq m} \frac{a_n}{n}<\infty.\]
Take $\epsilon>0$ and choose $k\geq m$ such that $a_k/k<L+\epsilon$.
For each $n\geq 2k$ we have that $n=:p_nk+q_n+m$,  where $p_n\geq 1$ and $ 0\leq q_n<k$, so that $a_n \leq a_{p_n k} +a_{q_n+m}\leq p_n a_k + a_{q_n+m}$. This gives the inequality 
\[
\frac{a_n}{n} \leq \frac{p_n k}{n}\frac{a_k}{k}+\frac{a_{q_n+m}}{n}.
\]
Letting $n\to\infty$ shows that 
\[
\limsup_{n\to\infty}\frac{a_n}{n}\leq \frac{a_k}{k},
\]
since $p_nk/n\to 1$ as $n\to\infty$, and $a_{j+m}<\infty$ for all $0\leq j<k$. 
Thus, 
\[
L\leq\liminf_{n\to\infty}\frac{a_n}{n}\leq \limsup_{n\to\infty}\frac{a_n}{n}\leq \frac{a_k}{k}\leq L+\epsilon, 
\]
which shows that $\lim_{n\to\infty}a_n/n=L$.

Write $r:= \lim_{k\to\infty} \|f^k\|^{1/k}_{C\cap K_u}$ and $r_k:=\|f^k\|^{1/k}_{C\cap K_u}$. 
It is an easy calculus exercise to show that 
\[
\lim_{k\to\infty} r_{k+n}^{(k+n)/k}=r\mbox{\quad  for all $n\geq 1$}.
\] 

Let $\kappa >0$ be the normality constant of $K$. If $x\in C\cap K_u$ and $\|x\|\leq 1$, then $f^m(x)\leq_K Mu$, so that 
\[f^{k+m}(x)\leq_K Mf^k(u),\] 
which gives
\[
\|f^{k+m}(x)\|\leq M\kappa\|f^k(u)\|
\]
for all $x\in C\cap K_u$ with $\|x\|\leq 1$. So, 
\[
r_{m+k}^{(m+k)/k}\leq (M\kappa \|f^k(u)\|)^{1/k}\leq (M\kappa)^{1/k}r_k,
\]
and hence $\lim_{k\to\infty} \|f^k(u)\|^{1/k}=r$. 
\end{proof}

It is well known; see for example \cite[Theorem 2.38]{AT} or \cite[Theorem 4.4]{Kras}, that as $K$ is normal, $(X,\|\cdot\|)$ admits an equivalent {\em monotone norm} $|\cdot |$, i.e., $|x|\leq|y|$ whenever $0\leq_Kx\leq_K y$.  Given a homogeneous mapping $f\colon C\cap K_u\to C\cap K_u$ which is order-preserving with respect to $K$ and $\epsilon>0$, define $f_{\epsilon, u}\colon C\cap K_u\to C\cap K_u$ by 
\begin{equation}\label{eq:4.1}
f_{\epsilon,u}(x):= f(x) +\epsilon |x|u\mbox{\quad for }x\in C\cap K_u. 
\end{equation}
Note that $f_{\epsilon,u}$ is homogeneous and order-preserving with respect to $K$, as $|\cdot|$ is a monotone norm. Moreover, for each $x\in C\cap K_u$ we have that $f(x)\leq_K f_{\epsilon,u}(x)$ and 
\[
\sup_{x\in C\cap K_u\colon |x|\leq 1} |f(x) -f_{\epsilon,u}(x)|\leq \epsilon |u|.
\]
\begin{theorem}\label{thm:4.1.2}
Let $f\colon C\cap K_u\to C\cap K_u$ be a homogeneous mapping which is order-preserving with respect to $K$, and let $f_{\epsilon,u}\colon C\cap K_u\to C\cap K_u$ be given by (\ref{eq:4.1}). If $f$ is $u$-bounded, then the following assertions hold:
\begin{enumerate}[(i)]
\item For each $\epsilon >0$, the mapping $f_{\epsilon,u}$ has a unique eigenvector $v_{\epsilon,u}\in C\cap K_u$ with $|v_{\epsilon,u}|=1$ and 
\[
f_{\epsilon,u}(v_{\epsilon,u})=:r_{\epsilon,u}v_{\epsilon,u}. 
\]
\item The mapping $\epsilon\mapsto v_{\epsilon,u}$ is continuous in the norm topology for $\epsilon>0$.
\item For each $\epsilon>0$, 
\[
r_{\epsilon,u}=\lim_{k\to\infty}\|f^k_{\epsilon,u}\|_{C\cap K_u}^{1/k} = 
\inf_{k\geq 1}\|f^k_{\epsilon,u}\|_{C\cap K_u}^{1/k}= 
\lim_{k\to\infty}\|f^k_{\epsilon,u}(u)\|^{1/k}
\]
and 
\[ 
\lim_{\epsilon \to 0^+} r_{\epsilon,u} = \lim_{k\to\infty}\|f^k\|_{C\cap K_u}^{1/k}.
\] 
\end{enumerate}
\end{theorem}
\begin{proof}
For notational convenience write $f_\epsilon:=f_{\epsilon,u}$, $r_\epsilon:=r_{\epsilon,u}$, and $v_\epsilon:=v_{\epsilon,u}$. Let $\Sigma :=\{x\in C\cap K_u\colon |x|=1\}$ and define $g_\epsilon\colon \Sigma\to\Sigma$ by 
\[
g_\epsilon(x):=\frac{f_\epsilon(x)}{|f_\epsilon(x)|}\mbox{\quad for }x\in\Sigma.
\]
As $\epsilon u\leq_K f_\epsilon(x)$ for all $x\in\Sigma$, 
$\epsilon |u|=|\epsilon u|\leq |f_\epsilon(x)|$ for all $x\in\Sigma$. 

Using the equivalence of the norms $|\cdot |$ and $\|\cdot\|$ and the fact that $f$ is $u$-bounded, we know that there exists $M_1>0$ such that 
\[
f(x)\leq_K M_1|x|u\mbox{\quad for all }x\in C\cap K_u.
\]
This implies for $x\in\Sigma$ that $f_\epsilon(x)\leq_K (M_1+\epsilon)u$ and $|f_\epsilon(x)|\leq (M_1+\epsilon)|u|$. 

By definition of Hilbert's metric $\delta_K$ we know that 
\[
\delta_K(g_\epsilon(x),g_\epsilon(y))=\delta_K(f_\epsilon(x),f_\epsilon(y))
\]
for all $x,y\in \Sigma$, and 
\begin{equation}\label{eq:4.1.2}
\delta_K(g_\epsilon(x),u)=\delta_K(f_\epsilon(x),u)\leq \log \left (\frac{M_1+\epsilon}{\epsilon}\right )
\end{equation}
for all $x\in\Sigma$. 

Let $\Gamma:=\{x\in K_u\colon |x|=1\}$. We know; see \cite[Theorem 4.8]{Kras}, that the metric space $(\Gamma,\delta_K)$ is complete, as $K$ is a closed normal cone in $(X,\|\cdot\|)$.  We will now show that $\Sigma$ is a closed subset of $(\Gamma,\delta_K)$, from which we conclude that $(\Sigma,\delta_K)$ is also a complete metric space. 

Suppose that $(x_k)_k$ is a sequence in $\Sigma$ converging to $\xi$ in $(\Gamma,\delta_K)$, Thus, there exist $0<\alpha_k\leq\beta_k$ such that $\alpha_k\xi\leq_K x_k\leq_K\beta_k\xi$ for $k=1,2,\ldots$ and $\log (\beta_k/\alpha_k)\to 0$ as 
$k\to\infty$.  Since
\[
\alpha_k =|\alpha_k\xi|\leq |x_k|=1\leq |\beta_k \xi|=\beta_k, 
\]
we know that $\alpha_k\leq 1\leq \beta_k$, and hence $\lim_{k\to\infty} \alpha_k=1=\lim_{k\to\infty}\beta_k$. This implies that 
\[
0\leq_k x_k-\alpha_k\xi\leq_K(\beta_k-\alpha_k)\xi,
\]
so that $|x_k-\alpha_k\xi|\leq\beta_k-\alpha_k$, which shows that $\lim_{k\to\infty}|x_k-\xi|=0$.
Since $\|\cdot\|$ and $|\cdot|$ are equivalent, $C$ is closed in $(X,|\cdot|)$, and therefore $\xi\in C$. But also $|\xi|=1$, which shows that $\xi\in\Sigma$ and hence $\Sigma$ is closed in $(\Gamma,\delta_K)$. 

To proceed we fix $\epsilon_1>0$. Define 
\begin{equation}\label{eq:4.1.3} 
R:=\log \left ( \frac{M_1+\epsilon_1}{\epsilon_1}\right)>0\mbox{\quad and \quad}B_R:=\{x\in\Sigma\colon \delta_K(x,u)\leq R\},
\end{equation}
which is a closed subset of $(\Sigma,\delta_K)$, so $(B_R,\delta_K)$ is a complete metric space. 
Note that it follows from (\ref{eq:4.1.2}) that $g_\epsilon(B_R)\subseteq B_R$ for $\epsilon_1<\epsilon$. 

To prove that $f_\epsilon$ has a unique normalised eigenvector $v_\epsilon\in C\cap K_u$, it suffices to show that $g_\epsilon$ has a unique fixed point in $B_R$, where we chose $0<\epsilon_1<\epsilon$. The idea is to prove that $g_\epsilon$ is a contraction mapping on the complete metric space $(B_R,\delta_K)$. In other words, we will show that there exists $0\leq c<1$ such that for each $\epsilon>0$ with $\epsilon_1<\epsilon$, we have that 
\begin{equation}\label{eq:4.1.4}
\delta_K(g_\epsilon(x),g_\epsilon(y))=\delta_K(f_\epsilon(x),f_\epsilon(y))\leq c \delta_K(x,y)
\end{equation}
 for all $x,y\in B_R$. 
 
 To prove (\ref{eq:4.1.4}) we let $x\neq y$  in $B_R$ and define $\alpha:=M(x/y)^{-1}$ and $\beta:=M(y/x)$, so $\alpha x\leq_K y\leq_K\beta x$ and $\delta_K(x,y)=\log(\beta/\alpha)>0$. As $x,y\in\Sigma$, $\alpha=\alpha |x|\leq |y|=1\leq \beta |x|=\beta$, so that $\alpha\leq 1\leq \beta$. Now $\alpha\leq 1$ and $f(x)\leq_KM_1u$ give the inequality  
 \[
 \epsilon(1-\alpha)f(x)\leq_K\epsilon(1-\alpha)M_1u.
 \]
 Combining this with the inequalities $\alpha M_1f(x)\leq_K M_1f(y)$ and $\alpha\epsilon f(x)\leq_K\epsilon f(y)$ gives, 
\[
(\alpha M_1+\epsilon)(f(x)+\epsilon u)\leq_K (M_1+\epsilon)(f(y)+\epsilon u),
\]
which shows that 
\[
\alpha'f_\epsilon(x)\leq_K f_\epsilon(y),\mbox{\quad where }\alpha':=\frac{\alpha M_1+\epsilon}{M_1+\epsilon}.
\]
In a similar way it can be shown that 
\[
f_\epsilon(y)\leq_K \beta'f_{\epsilon}(x),\mbox{\quad where }\beta':=\frac{\beta M_1+\epsilon}{M_1+\epsilon}.
\]

So, if we let $\epsilon':=\epsilon/M_1$, we get that 
\begin{equation}\label{eq:4.1.5}
\delta_K(f_\epsilon(x),f_\epsilon(y))\leq \log\left (\frac{\beta'}{\alpha'}\right )=
 \log\left (\frac{\beta+\epsilon'}{\alpha+\epsilon'}\right ).
\end{equation}
Note that $\delta_K(x,y)\leq 2R$, so that $\log (\beta/\alpha)\leq 2R$, and hence 
$\beta/\alpha\leq e^{2R}=\left (\frac{M_1+\epsilon_1}{\epsilon_1}\right)^2$, as $x,y\in B_R$. Thus, to prove (\ref{eq:4.1.4}) it suffices to show that there exists $0\leq c<1$  such that for all $0<\alpha\leq 1\leq \beta$ with $1<\beta/\alpha\leq e^{2R}$, and for all $\epsilon>0$ with $\epsilon_1<\epsilon$ we have that 
\begin{equation}\label{eq:4.1.6}
\log\left (\frac{\beta+\epsilon'}{\alpha+\epsilon'}\right )\leq c\log\left ( \frac{\beta}{\alpha}\right ).
\end{equation}

Basic algebra gives
\[
\log\left (\frac{\beta+\epsilon'}{\alpha+\epsilon'}\right ) = \log\left (\frac{\beta}{\alpha}\right ) + \log\left ( 1-\left (\frac{\epsilon'}{\alpha+\epsilon'}\right )\left (1-\frac{\alpha}{\beta}\right)\right ).
\]
Writing $\rho:=\beta/\alpha$, so $1<\rho\leq e^{2R}$, and using the fact that $0<\alpha\leq 1$, we derive that 
\begin{equation}\label{eq:4.1.7} 
\log\left (\frac{\beta+\epsilon'}{\alpha+\epsilon'}\right ) 
 \leq 
\log\rho  + \log\left ( 1-\left (\frac{\epsilon'}{1+\epsilon'}\right )\left (1-\frac{1}{\rho}\right)\right ).
\end{equation}
Let $\gamma:=\epsilon'/(1+\epsilon')$, and for $1<\rho<e^{2R}$ consider the continuous function $\rho\mapsto \psi(\rho)$, where 
\[
\psi(\rho) = \frac{\log\rho  + \log\left ( 1-\left (\frac{\epsilon'}{1+\epsilon'}\right )\left (1-\frac{1}{\rho}\right)\right )}{\log \rho} = 1+ \frac{\log\left (1 -\gamma(1-1/\rho)\right)}{\log \rho}.
\]

Thus, to establish (\ref{eq:4.1.6}) it suffices to find $0<\delta\leq 1$, which is independent of $\gamma = \epsilon'/(1+\epsilon') = \epsilon/(M_1+\epsilon)$ for $0<\epsilon_1<\epsilon$, such that 
\begin{multline}\label{eq:4.1.8} 
\sup\left\{ \frac{\log\left (1 -\gamma(1-1/\rho)\right)}{\log \rho}\colon 1<\rho\leq e^{2R}\right\} \\= \sup\left\{ \frac{\log\left (1 -\gamma(1-e^{-\sigma})\right)}{\sigma}\colon 0<\sigma\leq 2R\right\}\leq -\delta.
\end{multline}
As $0<\gamma(1-e^{-\sigma})<1$, we can use Taylor's formula to get 
\[
\frac{1}{\sigma}\log(1-\gamma(1-e^{-\sigma})) =-\sum_{j=1}^\infty \frac{(\gamma(1-e^{-\sigma}))^j}{\sigma j} \leq -\frac{\gamma(1-e^{-\sigma})}{\sigma }
\]
for $0<\sigma\leq 2R$. Now consider the derivative of  $\theta\colon \sigma\mapsto (1-e^{-\sigma})/\sigma$: 
\[
\theta'(\sigma) = \frac{e^{-\sigma}(\sigma+1-e^\sigma)}{\sigma^2}.
\] 
As $e^\sigma>1+\sigma$ for all $\sigma>0$, we conclude that $\theta'(\sigma)<0$ on $0<\sigma\leq 2R$, and hence 
\[
\frac{1}{\sigma}\log(1-\gamma(1-e^{-\sigma}))\leq -\gamma\left (\frac{1-e^{-2R}}{2R}\right ).
\]
So, if we let 
\[
\delta:= \left (\frac{\epsilon_1}{M_1+\epsilon_1}\right )\left (\frac{1-e^{-2R}}{2R}\right )<1,
\]
then (\ref{eq:4.1.8}) holds for all $\epsilon_1<\epsilon$, as $\epsilon_1/(M_1+\epsilon_1)\leq \epsilon/(M_1+\epsilon)$ for all $\epsilon _1\leq \epsilon$. 

It now follows from (\ref{eq:4.1.7})  that $g_\epsilon$ is a contraction mapping on the complete metric space $(B_R,\delta_K)$ with contraction constant $1-\delta$ for all $\epsilon_1<\epsilon$. So, by the contraction mapping theorem $g_\epsilon$, $\epsilon_1<\epsilon $ has a unique fixed point $v_\epsilon\in \Sigma$. Moreover, $v_\epsilon$ is the unique normalised eigenvector in $\Sigma$ of $f_\epsilon$ and 
\[
f_\epsilon(v_\epsilon)=|f_\epsilon(v_\epsilon)|v_\epsilon.
\]
Writing $r_\epsilon:= |f_\epsilon(v_\epsilon)|$ and recalling that $|f_\epsilon(x)|\geq \epsilon |u|$ for all $x\in \Sigma$, we see that $r_\epsilon>\epsilon |u|$. 

We will now prove the second assertion, which is a consequence of the contraction mapping theorem.  Indeed, for $\epsilon_1>0$, let $R$ be as in (\ref{eq:4.1.3})  above. We know that for $\epsilon >\epsilon_1$, that the mapping $g_\epsilon$ is a contraction mapping on $(B_R,\delta_K)$ with contraction constant $0\leq c<1$, where $c$ is independent of $\epsilon$. This implies for $\mu,\epsilon>\epsilon_1$ that 
\[
\delta_K(v_\mu,v_\epsilon)\leq \frac{1}{1-c}\delta_K(g_\mu(v_\epsilon),v_\epsilon) = 
	\frac{1}{1-c}\delta_K(f_\mu(v_\epsilon),v_\epsilon).
\]
As $v_\epsilon\in B_R$, there exist $0<a\leq 1\leq b$ such that $a v_\epsilon\leq_K u\leq_K b v_\epsilon$. So, if  $\mu>\epsilon>\epsilon_1$, then 
\[
(1+a(\mu-\epsilon))v_\epsilon\leq_K f_\mu(v_\epsilon) = f_\epsilon(v_\epsilon) + (\mu-\epsilon)u\leq_K (1+b(\mu-\epsilon))v_\epsilon.
\]
This implies that 
\[
\delta_K(v_\mu,v_\epsilon)\leq \frac{1}{1-c}\delta_K(f_\mu(v_\epsilon),v_\epsilon)\leq 
\frac{1}{1-c}\log\left ( \frac{1+b(\mu-\epsilon)}{1+a(\mu-\epsilon)}\right )
\]
for $\epsilon_1<\epsilon<\mu$. As similar argument shows that 
\[
\delta_K(v_\mu,v_\epsilon)\leq 
\frac{1}{1-c}\log\left ( \frac{1+a(\mu-\epsilon)}{1+b(\mu-\epsilon)}\right )
\]
for $\epsilon_1<\mu<\epsilon$. So, if $\epsilon_k\to \mu$, then $\delta_K(v_{\epsilon_k},v_\mu)\to 0$, so that $\|v_{\epsilon_k}-v_\mu\|\to 0$, as the topology of $\delta_K$ is the same as the topology of $\|\cdot\|$ on $B_R$. 

To prove the third assertion we first note that we can apply Lemma \ref{lem:4.1.1} to $f_\epsilon$ and $f$ to get 
\begin{equation}\label{eq:4.1.9} 
\lim_{k\to\infty} \|f_\epsilon^k\|^{1/k}_{C\cap K_u} = \inf_{k\geq 1} \|f_\epsilon^k\|^{1/k}_{C\cap K_u} = \lim_{k\to\infty} \|f_\epsilon^k(u)\|^{1/k}
\end{equation}
and 
\begin{equation}\label{eq:4.1.10} 
\lim_{k\to\infty} \|f^k\|^{1/k}_{C\cap K_u} = \inf_{k\geq 1} \|f^k\|^{1/k}_{C\cap K_u} = \lim_{k\to\infty} \|f^k(u)\|^{1/k}.
\end{equation}

For $0\leq\epsilon\leq \mu$ it is easy to see that $f_\epsilon^k(x)\leq_K f^k_\mu(x)$ for all $x\in C\cap K_u$ and $k\geq 1$. Using the normality of $K$, we know that there exists a constant $M_2>0$ (independent of $k$) such that 
\[
\|f_\epsilon^k\|_{C\cap K_u} \leq M_2\|f^k_\mu\|_{C\cap K_u}
\]
for $0\leq \epsilon\leq \mu$. It follows that $\|f_\epsilon^k\|_{C\cap K_u}^{1/k} \leq \left(M_2\|f^k_\mu\|_{C\cap K_u}\right)^{1/k}$ for $0\leq\epsilon\leq\mu$, and hence
\[
\lim_{\epsilon\to 0^+}\left (\lim_{k\to\infty}\|f_\epsilon^k\|_{C\cap K_u}^{1/k} \right )\mbox{\quad exists }
\] 
and satisfies
\begin{equation}\label{eq:4.1.11}
\lim_{\epsilon\to 0^+}\left (\lim_{k\to\infty}\|f_\epsilon^k\|_{C\cap K_u}^{1/k} \right )\geq \lim_{k\to\infty}\|f^k\|^{1/k}_{C\cap K_u}.
\end{equation}

Recall that $f_\epsilon(v_\epsilon) =r_\epsilon v_\epsilon$ and $v_\epsilon\in\Sigma$. So, there exist $0<\lambda_1\leq\lambda_2$ (depending on $\epsilon$) such that $\lambda_1u\leq_K v_\epsilon\leq_K\lambda_2 u$. This implies that $\lambda_1f^k_\epsilon(u)\leq_K r_\epsilon^k v_\epsilon\leq_K\lambda_2 f^k_\epsilon(u)$ for all $k\geq 1$, and hence 
\begin{equation}\label{eq:4.1.12} 
r_\epsilon =\lim_{k\to\infty} \|f^k_\epsilon (u)\|^{1/k}.
\end{equation}

It remains to show that $\lim_{\epsilon\to 0^+} r_\epsilon =\lim_{k\to\infty}\|f^k(u)\|^{1/k}$. 
Combining (\ref{eq:4.1.9}--\ref{eq:4.1.12}) we see  that it suffices to show that 
\[
\lim_{\epsilon\to 0^+} r_\epsilon \leq \lim_{k\to\infty}\|f^k(u)\|^{1/k}=:r_{C\cap K_u}(f).
\]
First note that there exists $\gamma>0$, independent of $\epsilon\geq 0$, such that $u\leq_K\gamma f_\epsilon(u)$. For each $|x|\leq 1$ and $0<\epsilon<M_1$ we also know that 
\[
f_\epsilon(x)\leq_K (M_1+\epsilon) u\leq_K 2M_1 u.
\] 
Thus, for each $k\geq 1$ and $|x|\leq 1$ we have that $f^k_\epsilon(x)\leq_K 2M_1\gamma f^k_\epsilon(u)$.  As the norms $|\cdot |$ and $\|\cdot\|$ are equivalent on $X$, there exists a constant $M_3>0$ such that 
\begin{equation}\label{eq:4.1.13} 
\|f^k_\epsilon(x)\|\leq M_3\|f_\epsilon^k(u)\| 
\end{equation}
for all $x\in C\cap K_u$ and $\|x\|\leq 1$. 

Now fix $\eta >0$ and choose $N\geq 1$ so large that 
\begin{equation}\label{eq:4.1.14} 
M_3^{1/N}\|f^N(u)\|^{1/N}< r_{C\cap K_u}(f) + \eta/2.
\end{equation}
Since $\lim_{\epsilon\to 0^+}\|f_\epsilon^N(u)\|=\|f^N(u)\|$, there exists $\epsilon(\eta)>0$ such that for $0<\epsilon <\epsilon(\eta)$, 
\[
M_3^{1/N}\|f^N_\epsilon(u)\|^{1/N}< r_{C\cap K_u}(f) + \eta.
\]
From (\ref{eq:4.1.13}) we now deduce that 
\begin{eqnarray*}
\|f_\epsilon^N\|^{1/N}_{C\cap K_u} & = & \sup\{\|f^N_\epsilon(x)\|^{1/N}\colon x\in C\cap K_u\mbox{ and } \|x\|\leq 1\}\\
 & \leq &  M_3^{1/N}\|f_\epsilon^N(u)\|^{1/N}\\
 & < & r_{C\cap K_u}(f) +\eta
\end{eqnarray*} 
for $0<\epsilon<\epsilon(\eta)$. It now follows from (\ref{eq:4.1.9}) that 
\[
\lim_{k\to\infty}\|f^k_\epsilon\|^{1/k}_{C\cap K_u}\leq \|f^N_\epsilon\|^{1/N}_{C\cap K_u} < r_{C\cap K_u}(f)+\eta.
\]
As $\eta >0$ was arbitrary, we conclude that $\lim_{\epsilon\to 0^+} r_\epsilon = r_{C\cap K_u}(f)$ and we are done. 
\end{proof}
\begin{remark}
The general idea of using perturbations of $f$ like $f_{\epsilon,u}$ has  been exploited before; see for example \cite{AGN}, \cite[Lemma 2.1]{MN}, and \cite[Lemmas 3.2, 3.9 and 4.1]{NTMNA}. 
In particular,  the reader should compare Section 3 and 4 of \cite{NTMNA}, where results similar to Theorem \ref{thm:4.1.2} are established, and \cite[Lemma 7.6]{AGN} which provides a proof of \cite[Lemma 3.2]{NTMNA}. 

We also remark that if $w\in K_u$ and $f$ is a $u$-bounded homogeneous mapping which is order-preserving with respect to $K$, then there exists a constant $M_w>0$ such that $f(x)\leq_K M_w \|x\| w$ for all $x\in C\cap K_u$. Now for $\epsilon >0$ we can consider the mapping $f_{\epsilon,w}\colon C\cap K_u\to C\cap K_u$ given by $f_{\epsilon,w}(x) =  f(x) +\epsilon |x|w$ for all $x\in C\cap K_u$. Then $f_{\epsilon,w}$ has a unique eigenvector $v_{\epsilon,w}\in C\cap K_u$ with $\|v_{\epsilon,w}\|=1$. A slight variant in the proof of Theorem \ref{thm:4.1.2} shows that the mapping $(\epsilon,u)\mapsto v_{\epsilon,u}$ is norm continuous; see \cite[Lemma 4.1]{NTMNA} for related results.
\end{remark}

If $f\colon C\cap K_u\to C\cap K_u$ is  $u$-bounded, homogeneous, and order-preserving with respect to $K$, then we define the {\em partial spectral radius of $f$} by 
\[
r_{C\cap K_u}(f) :=\lim_{k\to\infty}\|f^k\|^{1/k}_{C\cap K_u}.
\]
Note that, as $\|f^{k+m}\|_{C\cap K_u} \leq \|f^k\|_{C\cap K_u}\|f^m\|_{C\cap K_u}$ for all $m,k\geq 1$, it follows from the sub-additive lemma that $r_{C\cap K_u}(f) =\inf_k\|f^k\|_{C\cap K_u}^{1/k}<\infty$. Using the notation from Theorem \ref{thm:4.1.2} we obtain the following immediate corollary, which we shall need later. 
 \begin{corollary}\label{cor:4.1.3} 
 If $f \colon C\cap K_u\to C\cap K_u$ is u-bounded, homogeneous and order-preserving with respect to $K$, then 
 \[
 \lim_{\epsilon\to 0^+} r_{C\cap K_u}(f_{\epsilon,u}) = r_{C\cap K_u}(f).
 \]
 \end{corollary}

Of particular interest to us will be the case where $K=C$ and $K_u=C^\circ$. In that case the partial spectral radius $r_{C^\circ}(f)$ satisfies a Collatz-Wielandt formula, which generalizes \cite[Corollary 37]{GV}; see also \cite{AGN} and \cite[Section 5.6]{LNBook}.
 
\begin{theorem}[Collatz-Wielandt formula I]\label{thm:4.1.4}
If $C$ is a closed normal  cone with nonempty interior in a Banach space $X$ and $f\colon C^\circ\to C^\circ$ is order-preserving and homogeneous, then 
\[
r_{C^\circ}(f)=\inf_{y\in C^\circ} M(f(y)/y).
\]
\end{theorem}
\begin{proof}
Let $y\in C^\circ$. and recall that, as $C$ is normal, the norms $\|\cdot\|$ and $\|\cdot\|_y$ are equivalent.  Note that for each $k\geq 1$ and $0\leq_C x\leq_C y$ we have that $\|f^k(x)\|_y\leq \|f^k(y)\|_y$, as $f$ is order-preserving. This implies that 
\[\|f^k\|_{y,C^\circ}:=\sup\{\|f^k(x)\|_y\colon x\in C^\circ\mbox{ with }\|x\|_y\leq 1\}= \|f^k(y)\|_y.\]
It now follows from Lemma \ref{lem:4.1.1} that,
\[
r_{C^\circ}(f) = \lim_{k\to\infty}\|f^k(y)\|_y^{1/k} = \inf_{k\geq 1}\|f^k(y)\|_y^{1/k}\leq M(f(y)/y),
\]
as $\|f(y)\|_y=M(f(y)/y)$, so that $r_{C^\circ}(f)\leq \inf_{y\in C^\circ}M(f(y)/y)$. 

Now let $\epsilon >0$, $u\in C^\circ$ and  $f_{\epsilon,u}$ be as in (\ref{eq:4.1}). 
So, $f(x)\leq_C f_{\epsilon,u}(x)$ for all $x\in C^\circ$ and $r_{C^\circ}(f_{\epsilon,u})\to r_{C^\circ}(f)$ as $\epsilon\to 0^+$ by Corollary \ref{cor:4.1.3}. 
We know from Theorem \ref{thm:4.1.2} that there exist $v_{\epsilon,u}\in C^\circ$ such that  $f_{\epsilon,u}(v_{\epsilon,u})=r_{C^\circ}(f_{\epsilon,u}) v_{\epsilon,u}$. Thus, $M(f_{\epsilon,u}(v_{\epsilon,u})/v_{\epsilon,u})=r_{C^\circ}(f_{\epsilon,u})$, so that 
\[
r_{C^\circ}(f) = \lim_{\epsilon\to 0^+} M(f_{\epsilon,u}(v_{\epsilon,u})/v_{\epsilon,u})\geq 
\liminf_{\epsilon\to 0^+} M(f(v_{\epsilon,u})/v_{\epsilon,u})\geq \inf_{y\in C^\circ} M(f(y)/y).
\]
\end{proof}

The following two basic observations concerning $r_{C^\circ}(f)$ will be useful to us later. The first one is essentially \cite[Lemma 2.2]{LN1}. For completeness we include the short proof.
 \begin{lemma}\label{lem:4.1.5} 
Let $C$ be a normal closed cone with nonempty interior in a Banach space $X$. If $f\colon C^\circ\to C^\circ$ is an order-preserving homogeneous mapping, then $r_{C^\circ}(f)>0$.  
\end{lemma}
\begin{proof}
Pick $u,x\in C^\circ$. As $f(x)\in C^\circ$, there exists $\alpha>0$ such that $\alpha x\leq_C f(x)$, so that $\alpha^kx\leq_C f^k(x)$ for all $k\geq 1$. We will show that $\alpha\leq  r_{C^\circ}(f)$. 

Suppose that  $\alpha > r_{C^\circ}(f)+\epsilon$ for some $\epsilon >0$. The definition of $r_{C^\circ}(f)$ implies that $\|f^k(x)\|\leq (r_{C^\circ}(f) +\epsilon)^k$ 
for all $k\geq 1$ sufficiently large. Since $\alpha > r_{C^\circ}(f)+\epsilon$, we conclude that  $\alpha^{-k} f^k(x)\to 0$ as $k\to\infty$. However, $\alpha^{-k}f^k(x) -x\in C$ for all $k\geq 1$, so that $-x\in C$, which is impossible. 
\end{proof}

\begin{lemma}\label{cor:orbitzero} 
Let $C$ be normal closed cone  with nonempty interior in a Banach space $X$. If $f\colon C^\circ\to C^\circ$ is an order-preserving homogeneous mapping with $r_{C^\circ}(f)=1$, then for each $x\in C^\circ$ the orbit $\mathcal{O}(x;f)$ does not accumulate at $0$. 
\end{lemma} 
\begin{proof}
Let  $u\in C^\circ$ and $v_{\epsilon,u}\in C^\circ$ be as in Theorem \ref{thm:4.1.2}. 
As $C$ is a normal cone, we know that the norms $\|\cdot\|$ and $\|\cdot\|_u$ are equivalent, and hence there exists a constant $M>0$ such that $\|v_{\epsilon,u}\|_u\leq M$ for all $0<\epsilon\leq 1$. If $x\in C^\circ$, then $u\leq_C \beta x$ for some $\beta>0$. So, if we let  $z:=\beta M x$, then $v_{\epsilon,u}\leq_C z$ for all $0<\epsilon\leq 1$. 

We will show that $\mathcal{O}(z; f)$ does not accumulate at $0$. 
As $f\colon C^\circ\to C^\circ$, we know that for each $n\geq 1$ there exists $\beta_n>0$  (only depending on $f^n(z)\in C^\circ$) such that $|f^{n-1}(z)|u\leq_C \beta_n f^n(z)$. Thus,
\[
f_{\epsilon,u}(f^{n-1}(z))\leq_C (1+\epsilon\beta_n)f^n(z)
\]
for all $n\geq 1$. Now fix $k\geq 1$ and note that,  as $v_{\epsilon,u}\leq_C z$,
\[v_{\epsilon,u}  \leq_C f_{\epsilon,u}^{k-1}(f_{\epsilon,u}(z)) 
\leq_C  (1+\epsilon\beta_1)f_{\epsilon,u}^{k-1}(f(z))
   \leq_C \ldots\leq_C \prod_{i=1}^k (1+\epsilon\beta_i)f^k(z).
\]
It follows that $1\leq \prod_{i=1}^k (1+\epsilon\beta_i)|f^k(z)|$. So, if we let $\epsilon\to 0$,  we see that $1\leq |f^k(z)|$. Thus, $\mathcal{O}(z;f)$ does not accumulate at $0$. 
As $f$ is homogeneous, it follows that no orbit inside $C^\circ$ can accumulate at $0$. 
\end{proof}

Later, in Theorem \ref{thm:wolff}, we shall need to assume that the set $\{v_{\epsilon,u}\colon 0<\epsilon\leq 1\}$ in Theorem \ref{thm:4.1.2} has a convergent subsequence in the norm topology, which is always the case in finite dimensional spaces, but not in infinite dimensions. For this reason we introduce the following terminology. 
\begin{definition}
Let $C$ be normal closed cone  with nonempty interior in a Banach space $X$. If $f\colon C^\circ\to C^\circ$ is an order-preserving homogeneous mapping and the set  $\{v_{\epsilon,u}\colon 0<\epsilon\leq 1\}$ in Theorem \ref{thm:4.1.2} has a convergent subsequence in the norm topology, then we say that {\em $f$ has converging approximate eigenvectors}.   
\end{definition}
In the next subsection we shall  establish several sufficient conditions for a mapping to have converging approximate eigenvectors using so called generalised measures of non-compactness or simply generalised MNC's. 

\subsection{Generalised measures of non-compactness}
Let $X$ be a (real or complex) Banach space and let $\mathcal{B}(X)$ denote the collection of all bounded, non-empty, subsets of $X$. Given $S,T\in\mathcal{B}(X)$ we let $\mathrm{co}(S)$ denote the convex hull of $S$, $S+T:=\{s+t\colon s\in S\mbox{ and } t\in T\}$, and $\lambda S:=\{\lambda s\colon s\in S\}$ for all $\lambda $ in the scalar field. 
Following the terminology from \cite{MN} we call a mapping $\beta \colon \mathcal{B}(X)\to [0,\infty)$ a {\em generalised homogeneous measure of non-compactness (MNC)} if it satisfies the following conditions: 
\begin{enumerate} 
\item [A1.] For all $S\in \mathcal{B}(X)$, $\beta(S) =0$ if and only if $\overline{S}$ is compact. 
\item[A2.] For all $S\in \mathcal{B}(X)$, with $S\subseteq T$, we have that $\beta(S)\leq \beta(T)$. 
\item[A3.] For all $S\in\mathcal{B}(X)$ and $x_0\in X$ we have that $\beta(S\cup\{x_0\})=\beta(S)$. 
\item [A4.] For all $S\in \mathcal{B}(X)$ we have that $\beta(\overline{S})=\beta(S)$. 
\item [A5.] For all $S\in\mathcal{B}(X)$ we have that $\beta(\mathrm{co}(S))=\beta(S)$. 
\item [A6.] For all $S,T\in\mathcal{B}(X)$ we have that $\beta(S+T)\leq \beta(S)+\beta(T)$. 
\item [A7.] For all $S\in\mathcal{B}(X)$ and all scalars $\lambda$ we have that $\beta(\lambda S)=|\lambda|\beta (S)$. 
\end{enumerate}
Property (A7) is called the {\em homogeneity property} of $\beta$. Some treatments of MNC's assume that $\beta$ satisfies the so-called {\em set additive property}:
\begin{enumerate}
\item[A8.] For all $S,T\in\mathcal{B}(X)$ we have that $\beta(S\cup T)=\max\{\beta(S),\beta(T)\}$. 
\end{enumerate}
However, we shall not assume (A8). 

A fundamental example is the {\em Kuratowski measure of non-compactness}, 
\[
\alpha(S) :=\inf\left \{\delta>0\colon S=\bigcup_{i=1}^n S_i\mbox{ with $\mathrm{diam}(S_i)\leq \delta$ for all  $1\leq i\leq n<\infty$}\right\}
\]
for $S\in\mathcal{B}(X)$. The Kuratowski MNC satisfies properties (A1)--(A8). Notice that (A1) and (A8) imply (A2) and (A3), but there are many interesting examples of generalised homogeneous MNC's that do not satisfy (A8). 

Using the generalised homogeneous MNC's we can formulate a condition under which the set $\{v_{\epsilon,u}\colon 0<\epsilon\leq 1\}$ in Theorem \ref{thm:4.1.2} has a compact norm closure. 
\begin{theorem}\label{thm:4.2.1} Let $f\colon C\cap K_u\to C\cap K_u$ be a homogeneous mapping which is order-preserving with respect to $K$ and $u$-bounded.
Let  $\{v_{\epsilon,u}\colon 0<\epsilon\leq 1\}$  be as in Theorem \ref{thm:4.1.2}. If $r_{C\cap K_u}(f)>0$ and  there exists a generalised homogeneous MNC $\beta$ such that for each $A\in\mathcal{B}(X)$ with $A\subseteq C\cap K_u$ and $\beta (A)>0$ we have that 
\[
\beta(f(A))< r_{C\cap K_u}(f)\beta(A),
\]
then $\{v_{\epsilon,u}\colon 0<\epsilon\leq 1\}$  has a compact closure in the norm topology. 
\end{theorem} 
\begin{proof}
For simplicity write $S:=\{v_{\epsilon,u}\colon 0<\epsilon\leq 1\}$ and $r:=r_{C\cap K_u}(f)>0$. 
It suffices to show that $\beta(S) =0$ by (A1). Define $g(x):=\frac{1}{r}f(x)$ for all $x\in C\cap K_u$. So, $\beta (g(A))<\beta(A)$ for all $A\in\mathcal{B}(X)$ with $A\subseteq C\cap K_u$ and $\beta(A)>0$ by (A7).  

As $|v_{\epsilon,u}|=1$ and $f_{\epsilon,u}(v_{\epsilon,u} )= f(v_{\epsilon,u}) +\epsilon u = r_\epsilon v_{\epsilon,u}$, where $r_{\epsilon}:= r_{C\cap K_u}(f_{\epsilon,u})$, we get that 
\[
g(v_{\epsilon,u}) +\frac{\epsilon}{r} u +(1-\frac{r_\epsilon}{r})v_{\epsilon,u}= v_{\epsilon,u}.
\]
Define $T:=\{\frac{\epsilon}{r}u +(1-\frac{r_\epsilon}{r})v_{\epsilon,u}\colon 0<\epsilon\leq 1\}$.
Note that Corollary \ref{cor:4.1.3} implies that $\lim_{\epsilon\to 0^+} \frac{r_\epsilon}{r} =1$ and hence 
\[
\lim_{\epsilon\to 0^+} \frac{\epsilon}{r} u + (1-\frac{r_\epsilon}{r})v_{\epsilon,u}=0.
\]

Thus, the mapping $\sigma\colon \epsilon \mapsto  \frac{\epsilon}{r} u + (1-\frac{r_\epsilon}{r})v_{\epsilon,u}$ is a norm continuous mapping on $[0,1]$ by Theorem \ref{thm:4.1.2}. This implies that $\overline{T}$ is compact, so that $\beta(T)=0$. Since $ S\subseteq g(S)+T$, we conclude from (A2) and (A6) that 
\[
\beta(S)\leq \beta(g(S) +T) = \beta(g(S))+\beta(T) =\beta(g(S))
\]
so that $\beta(S)=0$.
\end{proof}
Notice that we have only used properties (A1), (A2), (A6) and (A7) of $\beta$ in the proof of Theorem \ref{thm:4.2.1}.  Another sufficient condition is given in the following result. 
\begin{theorem}\label{thm:4.2.2} Let $f\colon C\cap K_u\to C\cap K_u$ be a homogeneous mapping which is order-preserving with respect to $K$,  $u$-bounded, and satisfies $r_{C\cap K_u}(f)=1$.
Let  $\{v_{\epsilon,u}\colon 0<\epsilon\leq 1\}$ be as in Theorem \ref{thm:4.1.2}. If   there exists a generalised homogeneous MNC $\beta$ such that 
\[
\liminf_{m\to\infty} \beta(f^m(V)) =0,
\]
where $V:=\{x\in C\cap K_u\colon |x|\leq 1\}$, and $f$ is uniformly continuous on $V$ in the norm topology, then $\{v_{\epsilon,u}\colon 0<\epsilon\leq 1\}$  has a compact closure in the norm topology. 
\end{theorem} 
\begin{proof}
Note that by Theorem \ref{thm:4.1.2}(ii) it suffices to prove that $\beta(\{v_{\epsilon,u}\colon 0<\epsilon\leq\epsilon_0\leq 1\})=0$, where $\epsilon_0>0$ can be arbitrary small. Now let $\eta>0$ be given. 

We first show that for each $m\geq 1$ and $\sigma>0$ there exists $\epsilon_0:=\epsilon_0(\sigma,m)>0$  such that 
\begin{equation}\label{eq:4.2.2.1}
|f^m(v_{\epsilon,u}) - v_{\epsilon,u}|\leq \sigma\mbox{\quad for all }0<\epsilon\leq \epsilon_0.
\end{equation}
For $m=1$ the assertion follows from the fact that 
\[
|f(v_{\epsilon,u}) - v_{\epsilon,u}|\leq |f_{\epsilon,u}(v_{\epsilon,u}) -\epsilon u - v_{\epsilon,u}|\leq | r_{C\cap K_u}(f_{\epsilon,u})v_{\epsilon,u} -\epsilon u - v_{\epsilon,u}|\to 0,
\]
as $\epsilon\to 0^+$, since $r_{C\cap K_u}(f_{\epsilon,u})\to r_{C\cap K_u}(f)=1$ by Corollary \ref{cor:4.1.3}. 

Now suppose the assertion holds for all $1\leq j<m$. As $f$ is uniformly continuous on $V$, we know that there exists $\delta>0$ such that 
\[
|f(x)-f(y)|\leq \sigma/4\mbox{\quad for all }x,y\in V\mbox{ with } |x-y|\leq \delta.
\]
As $f$ is homogenous, it follows that $|f(x)-f(y)|\leq \sigma/2$ for all $x,y\in C\cap K_u$ with $|x|,|y|\leq 2$ and  $|x-y|\leq 2\delta$. 

As $|v_{\epsilon,u}|=1$ for all $0<\epsilon\leq 1$, we can use the induction hypothesis  to find 
$\epsilon_0>0$ such that $|f^{m-1}(v_{\epsilon,u})-v_{\epsilon,u}|\leq 2\delta$ and $|f^{m-1}(v_{\epsilon,u})|\leq 2$ for all $0<\epsilon\leq\epsilon_0$. Using uniform continuity of $f$ we deduce that 
\[
|f^{m}(v_{\epsilon,u})-f(v_{\epsilon,u})| = |f(f^{m-1}(v_{\epsilon,u}))-f(v_{\epsilon,u})|\leq \sigma/2
\]
for all  $0<\epsilon\leq\epsilon_0$. Applying the induction hypothesis again, and possibly decreasing $\epsilon_0>0$, we may also assume that 
\[
|f(v_{\epsilon,u})-v_{\epsilon,u}|\leq \sigma/2
\]
for all $0<\epsilon\leq\epsilon_0$. Combining these inequalities gives 
\[
|f^m(v_{\epsilon,u}) -v_{\epsilon,u}|\leq |f(f^{m-1}(v_{\epsilon,u}) )- f(v_{\epsilon,u})|+|f(v_{\epsilon,u}) -v_{\epsilon,u}|\leq \sigma/2+\sigma/2\leq \sigma
\] 
for all  $0<\epsilon\leq\epsilon_0$. 

As $\liminf_{m\to\infty}\beta(f^m(V)) =0$, there exists $m_0\geq 1$ such that $\beta(f^{m_0}(V))\leq \eta/2$. Define $\Gamma_{\epsilon_0}:=\{f^{m_0}(v_{\epsilon,u})\colon 0<\epsilon\leq\epsilon_0\}$. Taking $\sigma= \frac{\eta}{2\beta(B_1(0))}$ in  (\ref{eq:4.2.2.1}), we find an $\epsilon_0>0$ such that 
\[
\{v_{\epsilon,u}\colon 0<\epsilon\leq\epsilon_0\}\subseteq \Gamma_{\epsilon_0}+ \left \{ x\in C\cap K_u\colon |x|\leq \frac{\eta}{2\beta(B_1(0))}\right \},
\]
where $B_1(0):=\{x\in X\colon |x|\leq 1\}$. 

This implies that 
\[
\beta(\{v_{\epsilon,u}\colon 0<\epsilon\leq\epsilon_0\})\leq \beta(\Gamma_{\epsilon_0})+ 
\frac{\eta}{2\beta(B_1(0))}\beta(B_1(0))\leq \eta/2+\eta/2=\eta,
\]
as $\Gamma_{\epsilon_0}\subseteq f^{m_0}(V)$. Thus, $\beta(\{v_{\epsilon,u}\colon 0<\epsilon\leq 1\})=0$, which completes the proof. 
\end{proof}

\begin{remark}
If $X$ is a Banach space, $\beta$ is a generalised  homogeneous MNC on $X$ and 
$g\colon X\to X$ is a bounded linear map, one can define $\beta(g)$ as in Theorem \ref{thm:4.2.1}, i.e., $\beta(g):=\inf \{c>0\colon \beta(g(A))\leq c \beta(A)\mbox{ for all bounded subsets $A$ of $X$}\}$. However, as follows from \cite[Theorem 8]{MN2}, it may happen that $\beta(g^m)=\infty$ for infinitely many positive integers $m$.
\end{remark}

\section{Horofunctions of Hilbert's metric } 
The horofunction boundary, which goes back to Gromov \cite{Gr}, is   known to be a useful tool to prove Denjoy-Wolff type theorems for fixed point free nonexpansive mappings on a variety of metric spaces; see \cite{GV,Ka,Li1,NTMNA}. We shall also exploit horofunctions here. In fact, we shall follow Walsh \cite{Wa}, who made  detailed study of the horofunction boundary of finite dimensional Hilbert's metric spaces, and use the so called Funk and reverse Funk (weak) metrics. 

Let $C$ be a closed cone with nonempty interior in a Banach space $X$. For $x,y\in C^\circ$ the {\em Funk (weak) metric} is given by 
\begin{equation}\label{eq:F}
\F_C(x,y) :=\log M(x/y).
\end{equation}
Likewise, for $x,y\in C^\circ$ the {\em reverse Funk (weak) metric} is given by
\begin{equation}\label{eq:RF}
\RF_C(x,y) :=\log M(y/x).
\end{equation}
Using this notation we see that Hilbert's (projective) metric satisfies 
\begin{equation}\label{eq:F+RF}
\delta_C(x,y) = \F_C(x,y)+\RF_C(x,y)
\end{equation}
and Thompson's metric satisfies 
\begin{equation}\label{eq:FvRF}
d_C(x,y) = \max\{\F_C(x,y),\F_C(y,x)\}
\end{equation}
for all $x,y\in C^\circ$.  

The reader can check that both the Funk metric and reverse Funk metric satisfy the triangle inequality on $C^\circ\times C^\circ$, but are clearly neither symmetric nor nonnegative functions. They are named after P. Funk who studied them in \cite{Funk} in connection with Hilbert's fourth problem; see  \cite{PaTr} for more details.

We have the following lemma. 
\begin{lemma}\label{lem:5.1} 
Let $C$ be a closed cone with nonempty interior in a Banach space $X$. For each $y\in C^\circ$ the functions $x\mapsto \F_C(x,y)$ and $x\mapsto \RF_C(x,y)$  are Lipschitz with constant 1 with respect to $d_C$ on $C^\circ$. 
\end{lemma}
\begin{proof}
For $x_1,x_2\in C^\circ$ we have that $x_1\leq_C M(x_1/x_2)x_2\leq_C M(x_1/x_2)M(x_2/y)y$, so that $M(x_1/y)\leq M(x_1/x_2)M(x_2/y)$. This implies that 
$
\F_C(x_1,y)\leq \F_C(x_1,x_2)+\F_C(x_2,y)$. 
Interchanging the roles of $x_1$ and $x_2$ gives 
$\F_C(x_2,y)\leq \F_C(x_2,x_1)+\F_C(x_1,y)$, 
so that 
$
|\F_C(x_1,y)-\F_C(x_2,y)|\leq d_C(x_1,x_2)$.
The argument for $\RF_C$ goes in a similar fashion.
\end{proof}
It follows from Lemma \ref{lem:5.1} and (\ref{eq:F+RF}) that for each $y\in C^\circ$, the function $x\mapsto \delta_C(x,y)$ is Lipschitz with constant 2 with respect to $d_C$ on $C^\circ$. 

The following lemma lists some basic properties of $\F_C$ that are immediate from the definition. 
\begin{lemma}\label{lem:5.2} 
Let $C$ be a closed cone with nonempty interior in a Banach space $X$. Then $\F_C$ has the following properties:
\begin{enumerate}
\item For $x,y \in C^\circ$, and  $\alpha, \beta > 0$ we have that 
\[ \F_C(\alpha x, \beta y) = \F_C(x,y) + \log \alpha - \log \beta.\]
\item If $x_1, x_2 \in C^\circ$ with $x_1 \leq_C x_2$, and $y \in C^\circ$, then 
\[\F_C(x_1, y) \leq_C \F_C(x_2, y).\] 
\item If $y_1, y_2 \in C^\circ$ with $y_1 \leq_C y_2$, and $x \in C^\circ$, then 
\[\F_C(x,y_2) \leq_C \F_C(x,y_1).\] 
\item If  $f\colon C^\circ\to C^\circ$ is an order-preserving homogeneous mapping, then 
\[ \F_C(f(x),f(y)) \leq \F_C(x,y)\mbox{\quad for all }x,y\in C^\circ.  \]
\end{enumerate}
\end{lemma}

Following Walsh \cite{Wa} we now define the horofunction boundaries for the Funk metric, RFunk metric, and $\delta_C$. Fix a {\em base point} $b\in C^\circ$ and let $\rho$ be either $\F_C$, $\RF_C$, or, $\delta_C$. Let $\mathcal{C}(C^\circ)$ denote the space of continuous functions from $(C^\circ,d_C)$ into $\mathbb{R}$, equipped with the topology of compact convergence (also called the topology of uniform convergence on compact sets); see \cite[\S 46]{Mun}. Define $i_\rho\colon C^\circ\to \mathcal{C}(C^\circ)$ as follows: For each $y\in C^\circ$ the function $i_\rho(y)\in\mathcal{C}(C^\circ)$ is given by 
\[
i_\rho(y)(x) := \rho(x,y)-\rho(b,y)\mbox{\quad for all }x\in C^\circ.
\]

Note that for  each $x,x'\in C^\circ$ we have that 
\[
|i_\rho(y)(x)-i_\rho(y)(x')|= |\rho(x,y)-\rho(x',y)|\leq 2d_C(x,x')
\]
for all $y\in C^\circ$ by Lemma \ref{lem:5.1}, and hence $i_\rho(C^\circ):=\{i_\rho(y)\colon y\in C^\circ\}$ is an equicontinuous family in $\mathcal{C}(C^\circ)$. 
Furthermore, if $\rho$ is $\F_C$ or $\RF_C$, then  for each $x\in C^\circ$ we have that $
|i_\rho(y)(x)|\leq d_C(x,b)$ for all $y\in C^\circ$ by Lemma \ref{lem:5.1}. Also if $\rho =\delta_C$, then for each $x\in C^\circ$ we have that 
$|i_\rho(y)(x)|\leq 2d_C(x,b)$ for all $y\in C^\circ$. Thus, for each  fixed $x\in C^\circ$ the set $\{i_\rho(y)(x)\colon y\in C^\circ\}$ has compact closure in $\mathbb{R}$. 
It now follows from  Ascoli's Theorem \cite[Theorem 47.1]{Mun} that $i_\rho(C^\circ)$ has compact closure in $\mathcal{C}(C^\circ)$ with respect to the topology of compact convergence. 

The boundary, $\overline{i_\rho(C^\circ)}\setminus i_\rho(C^\circ)$, is called the {\em horofunction boundary} and its elements are called {\em horofunctions}. 
Note that $i_\rho(\alpha y)=i_\rho(y)$ for all $\alpha >0$ and $y\in C^\circ$. Thus, if we let $S:=\{y\in C^\circ\colon \|y\|=1\}$, then $\mathcal{H}_\rho=\overline{i_\rho(S)}\setminus i_\rho(S)$.
For simplicity we shall write $i_F:=i_\rho$ and $\mathcal{H}_F:=\overline{i_F(C^\circ)}\setminus i_F(C^\circ)$ if $\rho =\F_C$. Likewise, we use notation $i_R$ and $\mathcal{H}_R$ for $\rho=\RF_C$, and $i_H$ and $\mathcal{H}_H$ for $\rho=\delta_C$.

On $\overline{i_\rho(C^\circ)}$ the topology of compact convergence  agrees with the topology of pointwise convergence. It also coincides with the compact open topology; see \cite[\S 46]{Mun}. If $C$ is a finite dimensional cone, the metric space $(C^\circ, d_C)$ is $\sigma$-compact, i.e.,  the union of countably many compact sets. In that case the topology of compact convergence on $\mathcal{C}(C^\circ)$ is metrizable, and hence each horofunction $h$ in $\mathcal{H}_\rho$ is the limit of a sequence 
$(i_\rho(y_n))_n$ where $(y_n)_n$ is in $C^\circ$. However, if $C$ is infinite dimensional $(C^\circ, d_C)$ is no longer $\sigma$-compact, and the topology of compact convergence is not metrizable. Therefore we shall  work with nets instead of sequences. So, for each $h\in\mathcal{H}_\rho$ there exists a net $(i_\rho(y_\alpha))_\alpha$ such that 
$i_\rho(y_\alpha)\to h$, where $y_\alpha\in C^\circ$ for all $\alpha$. Moreover, every net 
$(i_\rho(y_\alpha))_\alpha$ in $\overline{i_\rho(C^\circ)}$ has a convergent subnet, as 
$\overline{i_\rho(C^\circ)}$ is compact. 

The next lemma is an infinite dimensional version of \cite[Lemma 2.4]{Wa}. 
\begin{lemma}\label{lem:5.3} Let $C$ be a closed cone with nonempty interior in a Banach space $X$ and let $(i_R(y_\alpha))_\alpha$ be a net converging to $h\in \mathcal{H}_R$. If $(y_\alpha)_\alpha$ has a subnet converging to $y\in C\setminus\{0\}$ in the norm topology, then $y\in\partial C$ and 
\begin{equation}\label{eq:hr}
h(x)=\RF_C(x,y)-\RF_C(b,y)\mbox{\quad for all }x\in C^\circ.
\end{equation}
\end{lemma}
\begin{proof}
Let $(y_\beta)_\beta$ be a subnet of $(y_\alpha)_\alpha$ converging to $y\in C\setminus\{0\}$ in the norm topology. By Lemma \ref{lem:2.2} we know that $\RF_C(x,y_\beta)\to \RF_C(x,y)$ for all $x\in C^\circ$, and hence $i_R(y_\beta)$ converges to $x\mapsto \RF_C(x,y)-\RF_C(b,y)$, which proves (\ref{eq:hr}). Note also that, as $h\in \mathcal{H}_R$, the point $y\in\partial C$, as otherwise $h\in i_R(C^\circ)$.  
\end{proof}
In general, it appears to be difficult  to completely characterize $\mathcal{H}_F$. 
Instead, we observe that all Funk horofunctions have a kind of sub-gradient, which will prove useful later.  
\begin{lemma} \label{lem:5.6}
Let $C$ be a closed  cone with nonempty interior in a Banach space $X$. If $h\in\mathcal{H}_F$, then there exists $\varphi \in C^*\setminus \{0\}$ such that $\log \varphi(x) \le h(x)$ for all $x \in C^\circ$.  
\end{lemma}
\begin{proof}
Let $(i(y_\alpha))_\alpha$ be a net converging to $h\in\mathcal{H}_F$. 
For each $\alpha$ there exists $\phi_\alpha\in \Sigma^*_b$ such that 
$\F_C(b,y_\alpha) = \log \frac{\varphi_\alpha(b)}{\varphi_\alpha(y_\alpha)}$, by Lemma \ref{lem:2.2}. So, for each $\alpha$ and each  $x\in C^\circ$ we have that 
\begin{eqnarray*}
\F_C(x,y_\alpha) -\F_C(b,y_\alpha) & \geq&  \log \frac{\varphi_\alpha(x)}{\varphi_\alpha(y_\alpha)} - \log \frac{\varphi_\alpha(b)}{\varphi_\alpha(y_\alpha)}\\
&  = & \log \varphi_\alpha(x) - \log \varphi_\alpha(b)\\
 & = & \log \varphi_\alpha(x).
\end{eqnarray*}
As $\Sigma_b^*$ is weak*  compact, there is a subnet on which $\varphi_\alpha$ converges to a point $\varphi\in \Sigma_b^*$ in the weak* topology. 
Thus, $h(x) \geq \log \varphi(x)$ for all $x\in C^\circ$. 
\end{proof}

We shall also need the following fact. 
\begin{proposition}\label{prop:5.6.2} 
Let $(y_\alpha)_\alpha$ be a net in $C^\circ$ such that  $y_\alpha\to y\in\partial C\setminus\{0\}$. 
Then $i_R(y_\alpha)\to h_R\in\mathcal{C}(C^\circ)$, where $h_R(x)=\RF_C(x,y)-\RF_C(b,y)$ for all $x\in C^\circ$ and $h_R\in\mathcal{H}_R$.  If $(y_\beta)_\beta$ is a subnet of $(y_\alpha)_\alpha$, then $i_F(y_\beta)$ converges in $\mathcal{C}(C^\circ)$ if and only if $i_H(y_\beta)$ converges in $\mathcal{C}(C^\circ)$. Moreover, if $i_F(y_\beta)$ converges to $ h_F\in \mathcal{C}(C^\circ)$ and  $i_H(y_\beta)$ converges to $h_H\in \mathcal{C}(C^\circ)$, then $h_F\in\mathcal{H}_F$ and $h_H\in\mathcal{H}_H$.   
\end{proposition}
\begin{proof}
It follows from Lemma \ref{lem:2.2} that $i_R(y_\alpha)$ converges to  $ h_R\in\mathcal{C}(C^\circ)$, where $h_R(x)=\RF_C(x,y)-\RF_C(b,y)$ for all $x\in C^\circ$.  To show that $h_R\in\mathcal{H}_R$ we need to prove that there does not exist $v\in C^\circ$ such that 
\begin{equation}\label{eq:5.6.2.4}
h_R(x) =\RF_C(x,v)-\RF_C(b,v)
\end{equation}
for all $x\in C^\circ$. We argue by contradiction. So, suppose there exists $v\in C^\circ$ such that (\ref{eq:5.6.2.4}) holds. 
Let $(\epsilon_k)_k$ be a sequence of reals with $0<\epsilon_k<1$ and $\lim_{k\to\infty} \epsilon_k=0$.  Define 
$x_k:= \epsilon_kb+(1-\epsilon_k)y$ for all $k\geq 1$. Because $y\leq_C \left(\frac{1}{1-\epsilon_k}\right) x_k$, we see that $\log M(y/x_k)\leq -\log(1-\epsilon_k)$. As $y\neq 0$, we know that $-\infty<\log M(y/b)<\infty$, so that 
\begin{equation}\label{5.6.2.6}
\limsup_{k\to\infty} h_R(x_k)=\limsup_{k\to\infty} \RF_C(x_k,y)-\RF_C(b,y)<\infty.
\end{equation}
On the other hand, 
\[
\RF_C(x_k,v) = \log M(v/\epsilon_k b+(1-\epsilon_k)y)\to\infty
\]
as $k\to \infty$, because $v\in C^\circ$ and $\epsilon_k b+(1-\epsilon_k)y\to y\in\partial C$. Moreover, $\RF_C(b,v)$ is finite, as $b\in C^\circ$. So, if there exists $v\in C^\circ$ such that (\ref{eq:5.6.2.4}) holds, then
\begin{equation}\label{eq:5.6.2.5}
\lim_{k\to\infty}h_R(x_k)=\infty,
\end{equation}
which contradicts (\ref{5.6.2.6}). 

Now suppose that $(y_\beta)_\beta$ is a subnet of $(y_\alpha)_\alpha$. Then $i_R(y_\beta)$ still converges in $\mathcal{C}(C^\circ)$; and because $i_R(y_\beta) +i_F(y_\beta)=i_H(y_\beta)$, the convergence of $i_F(y_\beta)$ in $\mathcal{C}(C^\circ)$ is equivalent to the convergence of    $i_H(y_\beta)$ in $\mathcal{C}(C^\circ)$. 
Suppose that $i_F(y_\beta)$ converges to $ h_F\in \mathcal{C}(C^\circ)$ and  $i_H(y_\beta)$ converges to $h_H\in \mathcal{C}(C^\circ)$. It remains to show that $h_F\in\mathcal{H}_F$ and $h_H\in\mathcal{H}_H$.   

To prove that $h_F\in\mathcal{H}_F$, we need to show that there does not exist $v\in C^\circ$ such that 
\begin{equation}\label{eq:5.6.2.0}
h_F(x)=\F_C(x,v)-\F_C(b,v)
\end{equation}
for all $x\in C^\circ$.  Let $x_k$ be as above.  Note that for each $\beta$ we have that 
\begin{eqnarray*}
i_F(y_\beta )(x_k) &= &\F_C(x_k, y_\beta)-\F_C(b,y_\beta)\\
    &= & \log M(\epsilon_k b+(1-\epsilon_k)y/y_\beta) -\log M(b/y_\beta)\\
  &\leq & \log M(\epsilon_k b+(1-\epsilon_k)y/\epsilon_k b+(1-\epsilon_k)y_\beta)\\
  & & \qquad\qquad +\log M(\epsilon_k b+(1-\epsilon_k)y_\beta/y_\beta)  -\log M(b/y_\beta).
\end{eqnarray*}
We know that $y_\beta$ converges to $y$, so Lemma \ref{lem:2.2} implies that for each fixed $k\geq 1$, 
\[
M(\epsilon_k b+(1-\epsilon_k)y/\epsilon_k b+(1-\epsilon_k)y_\beta)\to 0.
\]
Also we have 
\[
 M(\epsilon_k b+(1-\epsilon_k)y_\beta/y_\beta) \leq \epsilon_k M(b/y_\beta)+(1-\epsilon_k)
\]
and $M(b/y_\beta)\to\infty$ as $y_\beta\to y\in\partial C$. Thus, 
\begin{equation}\label{eq:5.6.2.1}
i_F(y_\beta)(x_k) \leq M(\epsilon_k b+(1-\epsilon_k)y/\epsilon_k b+(1-\epsilon_k)y_\beta) +\log\left(\frac{\epsilon_k M(b/y_\beta)+(1-\epsilon_k)}{M(b/y_\beta)}\right).
\end{equation}
The right hand side of (\ref{eq:5.6.2.1}) converges to $\log(\epsilon_k)$ as $y_\beta\to y$, and hence $h_F(x_k)\leq\log(\epsilon_k)$ for $k\geq 1$.  Thus, 
\begin{equation}\label{eq:5.6.2.2}
\lim_{k\to\infty} h_F(x_k)=-\infty.
\end{equation}

On the other hand, if there exists a $v\in C^\circ$ such that (\ref{eq:5.6.2.0}) holds, then it follows from Lemma \ref{lem:2.2} that  
\begin{equation}\label{eq:5.6.2.3} 
\lim_{k\to\infty} h_F(x_k)= \log M(y/v)-\log M(b/v)> -\infty, 
\end{equation}
which contradicts (\ref{eq:5.6.2.2}).

If there exists $v\in C^\circ$ such that $h_H = i_R(v)+i_F(v)$,  the estimates in  (\ref{eq:5.6.2.5})  and (\ref{eq:5.6.2.3}) show  that 
\[
\lim_{k\to\infty} i_R(x_k) =\infty\mbox{\quad and\quad }\lim_{k\to\infty} i_F(x_k)> -\infty,
\]
which implies that 
\begin{equation}\label{final}
\lim_{k\to\infty}h_H(x_k)=\infty. 
\end{equation}
On the other hand, $h_H(x) = h_R(x)+h_F(x)$ for all $x\in C^\circ$. Equations  
(\ref{5.6.2.6}) and (\ref{eq:5.6.2.2}) show that 
\[
\limsup_{k\to\infty}h_R(x_k) <\infty.\mbox{\quad and\quad }
\lim_{k\to\infty} h_F(x_k) = -\infty,
\]
so $\lim_{k\to\infty} h_H(x_k)=-\infty$, which contradicts (\ref{final}) and shows that $h_H\in\mathcal{H}_H$. 
\end{proof}

Note that if $\rho$ is $\F_C$, $\RF_C$ or $\delta_C$, and $y\in C^\circ$, then 
\[
i_\rho(y)(x) =\rho(x,y)-\rho(b,y) =\rho(x,y/\|y\|)-\rho(b,y/\|y\|)
\]
for all $x\in C^\circ$. Thus, any horofunction is the limit of a net $(i_\rho(y_\alpha))_\alpha$ where $\|y_\alpha\|=1$ for all $\alpha$.  If $C$ is a finite dimensional cone, any sequence $(y_n)_n$ with $\|y_n\|=1$ for all $n$, has a limit point $y\in C$ with $\|y\|=1$. In that case it follows from Lemma \ref{lem:5.3} and Proposition \ref{prop:5.6.2} that
\[
\mathcal{H}_R=\{x\mapsto \RF_C(x,y)-\RF_C(b,y)\colon y\in\partial C\mbox{ and } \|y\|=1\}, 
\]
cf.\ \cite[Proposition 2.5]{Wa}.

\subsection{The horofunction boundary of a  symmetric cone}
If $C^\circ$ is a symmetric cone, there exists a particularly simple description of $\mathcal{H}_F$. Recall that a {\em symmetric cone} is the interior of the cone of squares in a Euclidean Jordan algebra.  A detailed exposition of the theory of symmetric cones can be found in \cite{FK} by Faraut and Kor\'anyi. We shall follow their notation and terminology. A {\em Euclidean Jordan algebra}, $(X,\bullet)$ is a finite dimensional real inner product space 
$(X,\langle\cdot,\cdot\rangle)$ equipped with a bilinear product $x\bullet y$ such that for each $x,y\in X$: 
\begin{enumerate}[(1)]
\item $x\bullet y= y\bullet x$, 
\item $x\bullet (x^2\bullet y) =x^2\bullet (x\bullet y)$, 
\item the linear map $L(x)\colon X\to X$ given by $L(x)w:=x\bullet w$ satisfies 
\[
\langle L(x)w,z\rangle = \langle w,L(x)z\rangle\mbox{ for all }w,z\in X.
\]
\end{enumerate}
The collection of squares in $(X,\bullet)$ forms a cone, $C$, and its interior is called a symmetric cone. 
We denote the unit element in $(X,\bullet)$ by $e$, which is an element of $C^\circ$. It is a basic consequence of the spectral decomposition theorem \cite[Theorem III.1.2]{FK} that $\|x\|_e := \inf\{\lambda>0\colon -\lambda e\leq_C x\leq_C\lambda e\}=\max\{|\lambda|\colon\lambda\in\sigma(x)\}$. 
 For $x\in X$ the linear mapping $P(x)\colon X\to X$ given by $P(x):=2L(x)^2-L(x^2)$ is called the {\em quadratic representation of $x$}. Note that $P(x^{-1/2})x =e$ for all $x\in C^\circ$. The mapping $P(x)$ maps the symmetric cone $C$  onto itself if $x\in X$ is invertible; see \cite[Proposition III.2.2]{FK}, and hence it preserves $\F_C$ by Lemma \ref{lem:5.2}. So, for $x,y\in C^\circ$, we have that  
 \[
 M(x/y) =M(P(y^{-1/2})x/e) =\max\{\lambda\colon \lambda\in \sigma(P(y^{-1/2})x)\},
 \]
 where the second equality  follows from  the spectral decomposition theorem \cite[Theorem III.1.2]{FK}. 
 \begin{theorem}\label{thm:5.4} 
 If $C^\circ$ is a symmetric cone in a Euclidean Jordan algebra $(X,\bullet)$ and we take the unit $e\in C^\circ$ as the base point to construct the horoboundaries, then the following assertions hold:
 \begin{enumerate}[(i)]
 \item $\mathcal{H}_F$ consists of those $f\in \mathcal{C}(C^\circ)$ for which there exists $z\in\partial C$ with $\|z\|_e=1$ such that $f(x)=\RF_C(x^{-1},z)$ for all $x\in C^\circ$. 
 \item $\mathcal{H}_R$ consists of those $g\in \mathcal{C}(C^\circ)$ for which there exists $y\in\partial C$ with $\|y\|_e=1$ such that $g(x)=\RF_C(x,y)$ for all $x\in C^\circ$. 
 \item $\mathcal{H}_H$ consists of those $h\in \mathcal{C}(C^\circ)$ for which there exist $y,z\in\partial C$ with $\|y\|_e=\|z\|_e=1$ and $y\bullet z=0$ such that 
 \[h(x)=\RF_C(x^{-1},z) + \RF_C(x,y)\mbox{\quad  for all $x\in C^\circ$.}\] 
 \end{enumerate}
 \end{theorem}
\begin{proof}
Let $g\in\mathcal{H}_R$ and let $(y_n)_n$ be a sequence in $C^\circ$, with $\|y_n\|_e=1$ for all $n$, such that $i_R(y_n)\to g$. By taking subsequences we may assume that $y_n\to y\in\partial C\setminus\{0\}$.  

It follows from Lemma \ref{lem:5.3} that 
\[
g(x)=\lim_{n\to\infty} i_R(y_n)(x) = \RF_C(x,y)-\RF_C(e,y)
\]
for all $x\in C^\circ$. But $\RF_C(e,y) = \log M(y/e) =\log \|y\|_e=0$, so that $g(x) = \RF_C(x,y)$ for all $x\in C^\circ$. On the other hand, if $y\in\partial C$ with $\|y\|_e=1$, then there exists a sequence $(y_n)_n$ in $C^\circ$ with $\|y_n\|_e=1$ for all $n$ such that $y_n\to y$. By taking subsequence we can also ensure that  $(i_R(y_n))_n$ converges to an element in $\mathcal{H}_R$. So,  by Lemma \ref{lem:5.3}
\[
\lim_{n\to\infty} i_R(y_n)(x) = \RF_C(x,y)-\RF_C(e,y)=\RF_C(x,y)
\]
for all $x\in C^\circ$, and hence $x\mapsto \RF_C(x,y)\in\mathcal{H}_R$. This completes the proof of part (ii). 

Let $f\in\mathcal{H}_F$ and let $(y_n)_n$ be a sequence in $C^\circ$, with $\|y_n\|_e=1$ for all $n$, such that $i_F(y_n)\to f$. By taking subsequences we may assume that $y_n\to y\in\partial C\setminus\{0\}$ and  $y_n^{-1}/\|y^{-1}_n\|_e\to z\in C$. 
Note that, as $y\in\partial C\setminus \{0\}$, it follows from the spectral decomposition theorem \cite[Theorem III.1.2]{FK} that $\|y^{-1}_n\|_e\to\infty$. This implies that 
\[
y\bullet z = \lim_{n\to\infty} y_n\bullet \left( \frac{y^{-1}_n}{\|y^{-1}_n\|_e}\right) =\lim_{n\to\infty} \frac{e}{\|y^{-1}_n\|_e}=0.
\]
It follows from  \cite[Exercise 3.3]{FK} that $\langle y,z\rangle =0$, and hence $z\in\partial C$, as $\langle v,w\rangle >0$ for all $w\in C^\circ$ and $v\in C\setminus\{0\}$. 

The inverse operation $w\mapsto w^{-1}$ on $C^\circ$ is known to be an order-reversing homogeneous of degree $-1$ involution; see \cite[Proposition 3.2]{Kai}.  This implies that $\F_C(u,v)=\RF_C(u^{-1},v^{-1})$ for all $u,v\in C^\circ$. Using Lemma \ref{lem:5.3} again we see that 
\begin{equation}\label{horobound}
\begin{split}
f(x) & =  \lim_{n\to\infty} \F_C(x,y_n)-\F_C(e,y_n)\\
    & =  \lim_{n\to\infty} \RF_C(x^{-1},y^{-1}_n/\|y^{-1}_n\|_e) - 
       \RF_C(e,y^{-1}_n/\|y^{-1}_n\|_e)\\
  & =  \RF_C(x^{-1},z)-\RF_C(e,z)\\
  & =  \RF_C(x^{-1},z)\\
\end{split}
\end{equation} 
for all $x\in C^\circ$, as $\RF_C(e,z)=\log \|z\|_e=0$.   

On the other hand, if $y,z\in\partial C$ with $\|y\|_e=\|z\|_e=1$ and $y\bullet z=0$, then there exists a Jordan frame $\{c_1,\ldots,c_k\}$ such that $y=\sum_{i=1}^p \lambda_ic_i$ and $z=\sum_{i=p+1}^q \mu_i c_i$ with $1=\lambda_1\geq \lambda_2\geq \ldots\geq \lambda_p>0$, $1=\mu_1\geq\mu_2\geq \ldots\geq \mu_q>0$, and $p<q\leq k$. 
For $n\geq 1$ define 
\begin{equation}\label{eq:y_n}
y_n:= \sum_{i=1}^p \lambda_ic_i + \sum_{i=p+1}^q \frac{1}{n^2\mu_i} c_i +\sum_{i=q+1}^k 
\frac{1}{n}c_i\in C^\circ .
\end{equation}
For sufficiently large $n$ we have that $\|y_n\|_e=1$ and 
\[
y_n^{-1}= \sum_{i=1}^p \frac{1}{\lambda_i}c_i + \sum_{i=p+1}^q n^2\mu_i c_i +\sum_{i=q+1}^k 
nc_i\in C^\circ .
\]
Note that $\|y_n^{-1}\|_e=n^2\mu_1=n^2$ for all large $n$, so that 
\[
\frac{y_n^{-1}}{\|y^{-1}_n\|_e}= \sum_{i=1}^p \frac{1}{n^2\lambda_i}c_i + \sum_{i=p+1}^q \mu_i c_i +\sum_{i=q+1}^k \frac{1}{n}c_i\in C^\circ 
\]
for all large $n$, which converges to $z$ as $n\to\infty$.  By taking a further subsequence we may assume that $i_F(y_n)$ converges to a point in $\mathcal{H}_F$. Using the same equations as in (\ref{horobound}) we see that $i_F(y_n)(x) \to  \RF_C(x^{-1},z)$ for all $x\in C^\circ$, which completes the proof of part (i). 

If $h\in \mathcal{H}_H$, then there exists a sequence $(y_n)_n$ in $C^\circ$ with $\|y_n\|_e=1$ for all $n$ such that $i_H(y_n)\to h$. By taking a subsequence we can assume that $y_n\to y\in\partial C$ and $y^{-1}_n/\|y^{-1}_n\|_e\to z\in C$. By the same argument as before we see that $y\bullet z =0$ and $z\in\partial C$. By taking a further subsequence we may also assume that $i_F(y_n)\to f\in\mathcal{H}_F$ and 
$i_R(y_n)\to g\in\mathcal{H}_R$, where $f(x) =\RF_C(x^{-1},z)$ and $g(x)=\RF_C(x,y)$ for all $x\in C^\circ$. This shows that  $h(x) = \RF_C(x^{-1},z)+\RF_C(x,y)$ for all $x\in C^\circ$.

To prove the other inclusion suppose that $y,z\in\partial C$ with $\|y\|_e=\|z\|_e=1$ and $y\bullet z=0$. Then we can define $y_n$ as in (\ref{eq:y_n}) for all $n\geq 1$. By taking a subsequence we can assume that $i_F(y_n)\to f\in\mathcal{H}_F$, $i_R(y_n)\to g\in\mathcal{H}_R$, and $i_H(y_n)\to h\in\mathcal{H}_H$. So, $h(x)=f(x)+g(x)$ for all $x\in C^\circ$. By the previous arguments $f(x) =\RF_C(x^{-1},z)$ and $g(x) = \RF_C(x,y)$ for all $x\in C^\circ$, which completes the proof. 
\end{proof}

\begin{remark}
If $C^\circ$  the symmetric cone of self-adjoint positive definite matrices over $\R$, $\C$ or $\H$, then for each $x,y\in C^\circ$ we have that 
\[ M(x/y) = \max \sigma(P(y^{-1/2})x) = \max \sigma(y^{-1/2}xy^{-1/2}) = \max \{\lambda\colon\lambda \in \sigma(y^{-1}x)\}. \]
So, in that case the horofunctions are given by 
\begin{enumerate}
 \item $h_F(x) = \log \max \sigma(xz)$,
 \item $h_R(x) = \log \max \sigma(x^{-1}y)$,
 \item $h_H(x) = \log \max \sigma(xz) + \log \max \sigma(x^{-1}y)$,
\end{enumerate}
where $y,z\in\partial C$ are such that $\|y\|_e=\|z\|_e=1$ and $y\bullet z=0$.

We also find an alternative way to describe the horofunctions of Hilbert's metric on the interior of the standard positive cone, 
$(\mathbb{R}^n_{+})^\circ=\{x\in\mathbb{R}^n\colon x_i>0\mbox{ for all }i\}$, then the one given in \cite{KMN}. Indeed, in that case Theorem \ref{thm:5.4} gives
\begin{enumerate}
 \item $h_F(x) = \log \max_i x_iz_i$,
 \item $h_R(x) = \log \max_i x^{-1}_iy_i$,
 \item $h_H(x) = \log \max_i x_iz_i + \log \max_i x^{-1}_iy_i$,
\end{enumerate}
where $y,z\in\partial \mathbb{R}^n_+$ are such that $\|y\|_\infty=\|z\|_\infty=1$ and $y_i z_i=0$ for all $i$. 
\end{remark}

\section{A Wolff type theorem for cones}
If $f\colon\Omega\to\Omega$ is fixed point free nonexpansive mapping on a finite dimensional Hilbert's metric space, then there exists a horofunction in $h\in \mathcal{H}_H$ such that  $h(f(x))\leq h(x)$ for all $x\in\Omega$; see \cite[Theorem 16]{GV} and  \cite[Theorem 3.4]{Ka}. The next theorem gives an analogous result for order-preserving homogenous mappings $f\colon C^\circ\to C^\circ$ that do not have an eigenvector in $C^\circ$, where the cone can be infinite dimensional.  
\begin{theorem} \label{thm:wolff}
Let $C$ be closed normal cone with nonempty interior in a Banach space $X$. If $f\colon C^\circ\to C^\circ$ is an order-preserving homogeneous mapping with no
eigenvector in $C^\circ$ and suppose that $f$ has converging approximate eigenvectors, then there exists a net $(v_\alpha)$ in $C^\circ$ with $v_\alpha\to v\in\partial C$ and $|v|=1$ such that  $i_F(v_\alpha)\to h_F\in\mathcal{H}_F $, $i_R(v_\alpha)\to h_R\in\mathcal{H}_R$ and  $i_H(v_\alpha)\to h_H\in\mathcal{H}_H$ with  $h_H(x) = h_F(x) + h_R(x)$ for all $x\in C^\circ$  such that 
\begin{enumerate}
\item $h_F(f(x)) \leq h_F(x) + \log r_{C^\circ}(f)$,
\item $h_R(f(x)) \leq h_R(x) - \log r_{C^\circ}(f)$,
\item $h_H(f(x)) \leq h_H(x)$.
\end{enumerate}   
Moreover, there exists $y \in \partial C\setminus\{0\}$ such that $h_R(x) = \RF_C(x,y)$ for all $x \in C^\circ$.  
\end{theorem}
\begin{proof}
Let  $u \in C^\circ$ be the base point to construct the horofunction boundaries. From Theorem \ref{thm:4.1.2} we know that for each   $\epsilon >0$ there exists $v_{\epsilon,u}\in C^\circ$  such that $|v_{\epsilon,u}|=1$ and
$f_{\epsilon,u}(v_{\epsilon,u})=r_{\epsilon,u} v_{\epsilon,u}$. For simplicity we shall write $v_\epsilon:=v_{\epsilon,u}$, $r_\epsilon:=r_{\epsilon,u}$, and $f_\epsilon:=f_{\epsilon,u}$.  It also follows from Theorem \ref{thm:4.1.2} that $r_\epsilon\to r_{C^\circ}(f)$ as $\epsilon \to 0$. 

As $f$ has converging approximate eigenvectors, we know that $\{v_\epsilon\colon 0<\epsilon< 1\}$ contains  convergent subsequence  $v_{\epsilon_n}$ with limit say $v\in C$. Note that $|v|=1$, and $v\in\partial C$, as otherwise $v$ is an eigenvector of $f$ in $C^\circ$.

From Proposition \ref{prop:5.6.2} we know that there exists a subnet $(v_{\epsilon_\alpha})$ such that $i_F(v_{\epsilon_\alpha})\to h_F\in\mathcal{H}_F$, $i_R(v_{\epsilon_\alpha})\to h_R\in\mathcal{H}_R$,  and $i_H(v_{\epsilon_\alpha})\to h_H\in\mathcal{H}_H$. By construction $h_H(x) =h_F(x)+h_R(x)$ for all $x\in C^\circ$.  
Thus, to prove the third inequality it suffices to show the first two. Using the third part of Lemma \ref{lem:5.2} we see that for each $\alpha$ and $x\in C^\circ$, 
\begin{eqnarray*} 
 \F_C(f(x),v_{\epsilon_\alpha}) - \F_C(u, v_{\epsilon_\alpha}) & \leq &  
 	\F_C (f_{\epsilon_\alpha}(x),v_{\epsilon_\alpha}) - \F_C(u, v_{\epsilon_\alpha})\\ 
	& = &  \F_C(f_{\epsilon_\alpha}(x),f_{\epsilon_\alpha}(v_{\epsilon_\alpha})) \\
	& & \qquad + 	          \log r_{\epsilon_\alpha}  - \F_C(u, v_{\epsilon_\alpha}) \\
	& \leq & \F_C(x,v_{\epsilon_\alpha}) + \log r_{\epsilon_\alpha} - \F_C(u, v_{\epsilon_\alpha}). 
	\end{eqnarray*} 
Thus, $h_F(f(x))\leq h_F(x)+\log r_{C^\circ}(f)$ for all $x\in C^\circ$. 	
	
To prove the second inequality fix $x\in C^\circ$, Note that as $f(x) \in C^\circ$, there exists a constant $\beta > 0$, depending on $x$, such that $|x| u \leq_C \beta f(x)$, and hence  
$(1+\beta\epsilon)^{-1} f_{\epsilon}(x) \leq_C f(x)$.  Using  Lemma \ref{lem:5.2}  we see that for each $\alpha$,  
\begin{eqnarray*}
 \RF_C(f(x),v_{\epsilon_\alpha}) - \RF_C(u, v_{\epsilon_\alpha}) & \leq &
	 \RF_C((1+\beta \epsilon_\alpha)^{-1}f_{\epsilon_\alpha}(x),v_{\epsilon_\alpha}) \\
	 & & - \RF_C(u, v_{\epsilon_\alpha})\\
	 &  = & \RF_C(f_{\epsilon_\alpha}(x),f_{\epsilon_\alpha} (v_{\epsilon_\alpha})) \\
	  & & - \RF_C(u, v_{\epsilon_\alpha}) + \log (1+\beta \epsilon_\alpha) \\
	  & & \qquad - \log r_{\epsilon_\alpha}\\ 	
	  & \leq & \RF_C(x,v_{\epsilon_\alpha}) - \RF_C(u, v_{\epsilon_\alpha}) \\
	  & & + \log (1+\beta \epsilon_\alpha) - \log r_{\epsilon_\alpha}.
	 \end{eqnarray*}
Thus, $h_R(f(x))\leq h_R(x)-\log r_{C^\circ}(f)$ for all $x\in C^\circ$. 

To prove the final assertion, note that by Lemma \ref{lem:5.3},  $h_R(x) = \RF_C(x,v)-\RF_C(u,v)$ for all $x \in C^\circ$.  Letting $y := M(v/u)^{-1}v$, we get that $h_R(x) = \RF_C(x,y)$.  
\end{proof}
The following example shows that equality can hold in the three inequalities in Theorem \ref{thm:wolff} simultaneously. 
\begin{example}\label{ex:par}   Consider the linear mapping 
\[ f(X) := MXM^*, \mbox{\quad where }M := \mat{1}{1}{0}{1}  \]
on the cone, $\Pi_2(\R)$, consisting  of positive semi-definite $2\times 2$ real matrices in the Jordan algebra of $2\times 2$ symmetric matrices. An elementary computation
shows for $k \geq 1$ that  
\begin{equation}\label{eq:f^k}
f^k(X) = \mat{a+2kb+k^2c}{b+kc}{b+kc}{c} \mbox{ \quad for }X = \mat{a}{b}{b}{c} \in \Pi_2(\R), \end{equation}
and hence $r_{\Pi_2(\R)^\circ}(f)=1$. 

Define the mapping $g$ on  $\Sigma^\circ$, consisting of invertible trace 1 matrices in $\Pi_2(\R)$, by $g(X) := f(X) / \mathrm{tr}(f(X))$. As $f$ is an invertible linear mapping from $\Pi_2(\R)$ onto itself, the mapping $g$ is an Hilbert metric isometry on $\Sigma^\circ$.
In fact, $g$ corresponds to a parabolic isometry of the hyperbolic plane. 
To see this, let 
\[ Y := \mat{1}{0}{0}{0}\mbox{\quad and\quad } Z := \mat{0}{0}{0}{1}, \]
By the above computation $g^k(X) \to Y$ for all $X \in \Sigma^\circ$, and hence $f$ has no eigenvector in $\Pi_2(\R)^\circ$. 
From Theorem \ref{thm:5.4} we know  there exist horofunctions $h_F(X)= \RF_{C}(X^{-1},Z)$ in $\mathcal{H}_F$, $h_R(X)= \RF_{C}(X,Y)$ in $\mathcal{H}_R$, and $h_H=h_F+h_R$ in $\mathcal{H}_H$, where we take the identity matrix $I$ as the base point. Note that 
\[
h_F(X) = \RF_{C}(X^{-1},Z) =  \log \max \sigma( XZ) =\log c  
\]
and 
\[
h_R(X) =\RF_{C}(X,Y)=  \log \max \sigma(X^{-1} Y) = \log(c/\det(X)).
\]
for all $X\in\Pi_2(\R)^\circ$. As $\det(f(X)) = \det(X)$, we deduce from (\ref{eq:f^k}) that 
\[h_F(f(X)) = \log c = h_F(X)\mbox{\quad  and\quad } h_R(f(X)) = \log(c/\det(X)) = h_R(X)\] for all $X\in \Pi_2(\R)^\circ$. 
Thus, for each $X\in \Pi_2(\R)^\circ$ we have that 
\[ h_H(f(X)) = \log c + \log (c/\det(X)) = h_H(X). \]
In Figure \ref{fig:hFandR} the level sets of $h_F$ and $h_R$ are depicted. 
\begin{figure}[h]
\begin{center}
\includegraphics[scale=0.5]{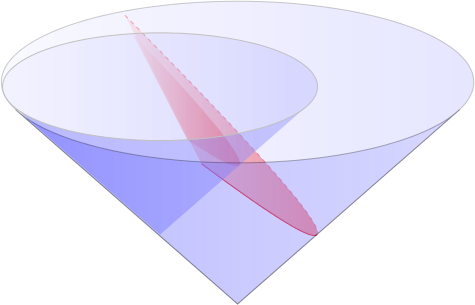} 
\end{center}
\caption[]{Funk and reverse-Funk horofunction level sets in $\Pi_2(\R)^\circ$.}
\label{fig:hFandR}
\end{figure}
\end{example}

The next corollary generalizes results from \cite{CFL} and \cite{GV} and is an immediate consequence of Lemma \ref{lem:5.6} and Theorem \ref{thm:wolff}.
\begin{corollary} \label{cor:weakupper}
If $C, X, f, y$ and $h_F$ are as in Theorem \ref{thm:wolff}, then the following assertions hold: 
\begin{enumerate} [(i)]
\item There exists $\varphi \in C^* \setminus \{ 0 \}$ such that $\log \varphi(f^k(x)) \leq h_F(x) + k \log r_{C^\circ}(f)$ for all $x \in C^\circ$ and $k \ge 1$.  
\item For all $x \in C^\circ$ such that $y \leq_C x$ we have that $r_{C^\circ}(f)y\leq_C f(x)$.  
\end{enumerate}
\end{corollary}
Another consequence of Theorem \ref{thm:wolff} concerns the linear escape rate studied in \cite{GV}. 
Recall that for an order-preserving homogeneous mapping $f\colon C^\circ\to C^\circ$ the 
{\em linear escape rate} is defined by 
\[
\rho(f) := \lim_{k\to\infty}  \frac{\RF_C(x,f^k(x))}{k}.
\]
Note that 
\[
 \frac{\RF_C(x,f^k(x))}{k} = \log M(f^k(x)/x)^{1/k} =\log \|f^k(x)\|^{1/k}_x \to \log r_{C^\circ}(f), \]
 as $k\to\infty$, so that 
 \[
 \rho(f) = \log r_{C^\circ}(f).
 \] 
The following characterization of $\rho(f)$ extends \cite[Theorem 1]{GV}.
\begin{corollary}\label{cor:6.4} 
Let $C$ be closed normal cone with nonempty interior in a Banach space $X$. If $f\colon C^\circ\to C^\circ$ is an order-preserving homogeneous mapping and $f$ has converging approximate eigenvectors, then 
\begin{equation}\label{eq:escape}
\rho(f) = \max_{h\in \mathcal{A}_R}\inf_{x\in C^\circ} h(x)-h(f(x)), 
\end{equation}
where $\mathcal{A}_R$ consists of those $h\in\overline{i_R(C^\circ)}$ for which there exists a net $(y_\alpha)$ in $C^\circ$, with $y_\alpha\to y\in C$ and $\|y\|_b=1$, such that $i_R(y_\alpha)\to h$ in $\mathcal{C}(C^\circ)$.

If $f$ has no eigenvector in $C^\circ$, then the maximum is attained at some $h\in \mathcal{A}_R\cap \mathcal{H}_R$. 
\end{corollary}
\begin{proof} Let $b\in C^\circ$ denote the base point  for the horofunctions. We know from Proposition \ref{prop:5.6.2} that for each element 
$h\in \mathcal{A}_R$ there exists $y\in C$ with $\|y\|_b=1$ such that 
\[
h(x) = \RF_C(x,y)-\RF_C(b,y)=\RF_C(x,y)\mbox{\quad for all }x\in C^\circ.
\]
So, 
\[
h(x) -h(f(x)) = \RF_C(x,y) -\RF_C(f(x),y)\leq \RF_C(x,f(x))\]
for all $x\in C^\circ$, and hence 
\[
\sup_{h\in \mathcal{A}_R} \inf_{x\in C^\circ} h(x)-h(f(x))\leq \inf_{x\in C^\circ} \RF_C(x,f(x)) =  \log r_{C^\circ}(f)
\]
by Theorem \ref{thm:4.1.4}.  

If $f$ has an eigenvector $v\in C^\circ$, then $f(v) =r_{C^\circ}(f)v$ and $h(v) -h(f(v)) = \log r_{C^\circ}(f)$ for all $h\in \overline{i_R(C^\circ)}$. 
As $\rho(f)=\log r_{C^\circ}(f)$, we see that the identity holds if $f$ has an eigenvector in $C^\circ$. 

If $f$ has no eigenvector in $C^\circ$, then we know from Theorem \ref{thm:wolff} that there exists $h_R\in\mathcal{A}_R\cap\mathcal{H}_R$ such that $\log r_{C^\circ}(f) \leq h_R(x)-h_R(f(x))$ for all $x\in C^\circ$.  Thus, if $f$ has no eigenvector in $C^\circ$, then  
\[
\rho(f) = \max_{h\in \mathcal{A}_R}\inf_{x\in C^\circ} h(x)-h(f(x)),
\]
which completes the proof.
\end{proof}
Note that if in Corollary \ref{cor:6.4} the cone $C$ is finite dimensional, then $\mathcal{A}_R = \overline{i_R(C^\circ)}$. 

Having established (\ref{eq:escape}) we can now use identical arguments as the ones used by Gaubert and Vigeral in \cite[Lemma 36 and Corollary 37]{GV} to obtain a second Collatz-Wielandt formula for $r_{C^\circ}(f)$, which generalises the one given in \cite[Corollary 37]{GV}.  The details are left to the reader. 
To formulate it we need to recall the following concept. Given an order-preserving homogenous mapping $f\colon C^\circ\to C^\circ$ on the interior of a closed cone in a finite dimensional vector space $X$, the \emph{radial extension of $f$} is given by 
\[
\hat{f}(x):=\lim_{\epsilon\to 0^+} f(x+\epsilon u)\mbox{\quad for all }x\in C,
\] 
where $u\in C^\circ$ is fixed. (It is easy to verify that $\hat{f}$ is an order-preserving  homogeneous mapping, and the limit exists  and is independent of $u\in C^\circ$, as $f$ is order-preserving and $C$ is finite dimensional.)
\begin{theorem}[Collatz-Wielandt formula II] \label{thm:cwII}
Let $C$ be closed cone with nonempty interior in a finite dimensional vector space $X$. If $f\colon C^\circ\to C^\circ$ is an order-preserving homogeneous mapping, then 
\[
r_{C^\circ}(f) = \max_{y\in C\setminus\{0\}} m(\hat{f}(y)/y), 
\]
where $m(\hat{f}(y)/y) :=\sup\{\alpha\geq 0\colon \alpha y\leq_C \hat{f}(y)\}$ for $y\in C\setminus\{0\}$. 
\end{theorem}
Theorem \ref{thm:cwII} should be compared with \cite[Corollary 5.4.2]{LNBook}, which implies that if $f\colon C\to C$ is a continuous, order-preserving, homogeneous mapping on a closed cone in a finite dimensional vector space $X$, then 
\[
r_C(f)=\max\{\alpha\geq 0\colon f(y)= \alpha y \mbox{ for some }y\in C\setminus\{0\}\}.
\]
The main difference is that in our case the mapping is only  defined on $C^\circ$, and need not have a continuous extension to the boundary; see \cite{BNS}. 

\section{Denjoy-Wolff theorems for Hilbert's metric} 
In this section we prove  Denjoy-Wolff type theorems for Hilbert's metric nonexpansive mappings on possibly infinite dimensional domains. We will consider  mappings $g\colon \Sigma^\circ\to\Sigma^\circ$ of the form:
\begin{equation}\label{eq:5.1}
g(x) =\frac{f(x)}{q(f(x))}\mbox{\quad for }x\in\Sigma^\circ:=\{x\in C^\circ\colon q(x)=1\},
\end{equation}
where $f\colon C^\circ\to C^\circ$ is an order-preserving homogeneous mapping on the interior of a normal closed $C$ in  a Banach space $X$ with $r_{C^\circ}(f)=1$ and 
$q\colon C^\circ\to (0,\infty)$ is a continuous homogenous mapping.  Typical examples of $q$ are  the norm (or an equivalent norm to the norm) of $X$ and strictly positive functionals $q\in C^\circ$.  Mappings $g$ of this form are nonexpansive under Hilbert's metric; see e.g. \cite[Section 2.1]{LNBook},
Note also that by Lemma \ref{lem:4.1.5} we can always renormalize $f$ so that $r_{C^\circ}(f)=1$ without changing $g$.  

\begin{theorem}\label{thm:7.1}
Let $C$ be normal closed cone with nonempty interior in a Banach space $X$  and  let $f\colon C^\circ\to C^\circ$ be a fixed point free order-preserving homogeneous mapping, with $r_{C^\circ}(f)=1$, satisfying the fixed point property on $C^\circ$ with respect to $d_C$.  Suppose that the mapping $g\colon\Sigma^\circ\to \Sigma^\circ$ is given by (\ref{eq:5.1}). If there exists $x_0\in C^\circ$ such that  $\mathcal{O}(x_0;f)$ and $\mathcal{O}(x_0/q(x_0);g)$  have compact closures in the norm topology,  then there exists a  convex set $\Omega\subseteq\partial C$ such that $\omega(z;g)\subseteq \Omega$ for all $z\in \Sigma^\circ$.
\end{theorem}
\begin{proof} 
Let $\Omega_0$ denote the convex hull of $\omega(x_0;f)$. The mapping $f\colon C^\circ\to C^\circ$ is nonexpansive under Thompson's metric, as it is order-preserving and homogeneous; see e.g. \cite[Section 2.1]{LNBook}. So, we obtain from 
 Corollary \ref{cor:3.3} that $\Omega_0\subseteq \partial C$. Using the Hahn-Banach separation theorem we find $\phi \in X^*$ such that $\Omega_0\subseteq \mathrm{ker}(\phi)$ and $\phi(z)>0$ for all $z\in C^\circ$. Now let $y_0:=x_0/q(x_0)$ and $\eta \in\omega(y_0;g)$. Then there exists a subsequence $(g^{k_i}(y_0))_i$ which converges to $\eta$. As $\mathcal{O}(x_0;f)$ has compact closure in the norm topology, we may assume, after taking a further subsequence, that $f^{k_i}(x_0)$ converges to say, $\xi$. It follows from Lemma \ref{cor:orbitzero}  that $\xi\neq 0$, and hence $q(\xi)>0$. So, 
\[
\phi(\eta) =\lim_{i\to\infty}\phi \left(\frac{f^{k_i}(x_0)}{q(f^{k_i}(x_0))}\right )= \lim_{i\to\infty}\frac{\phi(f^{k_i}(x_0))}{q(f^{k_i}(x_0))} = \frac{\phi(\xi)}{q(\xi)}=0,
\]
which shows that $\omega(y_0;g)\subseteq \mathrm{ker}(\phi)\cap C$.
As  $\mathcal{O}(x_0/q(x_0);g)$  has a compact closure in the norm topology, we can apply \cite[Theorem 5.3]{NTMNA} to conclude that $\cup_{z\in \Sigma^\circ} \omega(z;g)$ is contained in $\partial C$, which completes the proof. 
\end{proof}
\begin{remark} It is interesting to note that the assumption that $f\colon C^\circ\to C^\circ$ is a continuous order-preserving mapping such that for each $x\in C^\circ$ the orbit $\mathcal{O}(x;f)$ has a compact closure in the norm topology and all accumulation points of $\mathcal{O}(x;f)$  lie inside $\partial C$, is sufficient to prove that  $\omega(x;f)$ is contained in a convex subset of $\partial C$ for each $x\in C^\circ$. The argument goes as follows. 

Let $x\in C^\circ$ and note that $\omega(x;f)$ is a closed subset of $X$. As $\omega(x;f)$ is contained in the closure of $\mathcal{O}(x;f)$, which is compact, $\omega(x;f)$ is compact. Hence there exists $y\in C^\circ$ with $z\leq_C y$ for all $z\in\mathrm{cl}(\mathcal{O}(x;f))$. Indeed,  there exists $R>0$ such that $\omega(x;f)\subseteq B_R(0):=\{u\in X\colon \|u\|\leq R\}$. Now let $y_0\in C^\circ$. Then there exists $\delta>0$ such that $B_\delta(y_0):=\{u\in X\colon \|y_0-u\|\leq \delta\}\subseteq C$. If we let $y=\frac{R}{\delta}y_0$, then for each $z\in X$, with $\|z\|\leq R$, we have that 
\[
y-z= \frac{R}{\delta}(y_0 -\frac{\delta}{R}z) =: \frac{R}{\delta}(y_0 -z_0)\in C,  
\] 
 as $z_0 = \frac{\delta}{R} z\in X$ with $\|z_0\|\leq \delta$. 
 
 By assumption $\omega(y;f)$ is a compact subset of $\partial C$ and nonempty. 
 Let $w\in\omega(y;f)$. As $w\in\partial C$, there exists $\phi\in C^*\setminus\{0\}$ such that $\phi(w) =0$. 
 
 We now show that $z\leq_C w$ for all $z\in \omega(x;f)$. 
 If $(m_i)_i$ is such that $f^{m_i}(y)\to w$, and $(k_j)_j$ is such that $f^{k_j}(x)\to z$, then 
 \[
  f^{k_j}(x)\leq_C f^{m_i}(y)\mbox{\quad for all } k_j\geq m_i,
 \]
 as $f^k(x)\leq_C y$ for all $k\geq 0$. 
 Taking the limit for $j\to\infty$, we get that 
 \[
 z\leq_C f^{m_i}(y)\mbox{\quad for all  }m_i.
 \]
 Now letting $i\to\infty$, we find that $z\leq_C w$. As $\phi(w)=0$, we conclude that 
 $\phi(z)=0$ and hence $\omega(x;f)\subseteq\mathrm{ker}(\phi)\cap \partial C$. 
\end{remark}

\begin{theorem}\label{thm:7.2} 
Let $C$ be normal closed cone with nonempty interior in a Banach space $X$  and  let $f\colon C^\circ\to C^\circ$ be a fixed point free order-preserving homogeneous mapping, with  $r_{C^\circ}(f)=1$,  having converging approximate eigenvectors. Let $g\colon\Sigma^\circ\to \Sigma^\circ$ be given by (\ref{eq:5.1}), where $\Sigma^\circ=\{x\in C^\circ\colon q(x)=1\}$ is bounded in the norm topology. If there exists $x_0\in C^\circ$ such that $\lim_{k\to\infty} \|f^k(x_0)\|=\infty$ and the orbit 
$\mathcal{O}(x_0/q(x_0);g)$ has compact closure in the norm topology of $X$,
then there exists a  convex set $\Omega\subseteq\partial C$ such that $\omega(z;g)\subseteq \Omega$ for all $z\in \Sigma^\circ$.
\end{theorem}
\begin{proof} As $r_{C^\circ}(f)=1$ it follows from  Corollary \ref{cor:weakupper} that there exist $\psi\in C^*\setminus\{0\}$ and $h_F\in \mathcal{H}_F$ such that 
\begin{equation}\label{eq:upper}
\log \psi(f^k(x_0))\leq h_F(x_0)\mbox{\quad for all }k\geq 1.
\end{equation}
As $\Sigma^\circ$ is bounded in the norm topology, there exists $\delta>0$ such that $q(x)\geq \delta$ for all $x\in C^\circ$ with $\|x\|=1$. Indeed, if there exists a sequence $(u_k)_k$ in $C^\circ$ such that $\|u_k\|=1$ and $q(u_k)\leq 1/k$ for all $k$, then 
$u_k/q(u_k)\in \Sigma^\circ$, but $\|u_k/q(u_k)\|= 1/q(u_k)\to\infty$ as $k\to\infty$, which 
contradicts the fact that $\Sigma^\circ$ is bounded.
Combining this with the assumption, $\|f^k(x_0)\|\to\infty$ as $k\to\infty$, we find that 
\[
q(f^k(x_0)) =\|f^k(x_0)\|q\left(\frac{f^k(x_0)}{\|f^k(x_0)\|}\right)\geq \delta \|f^k(x_0)\|\to\infty\mbox{\quad as $k\to\infty$. }
\]
So, if we let $y_0:=x_0/q(x_0)$, then it follows from (\ref{eq:upper}) that 
\[
\psi(g^k(y_0)) = \frac{\psi(f^k(x_0))}{q(f^k(x_0))} \to 0\mbox{\quad as } k\to\infty.
\]
Thus, $\omega(y_0;g)\subseteq \mathrm{ker}(\psi)\cap C$. It now follows 
from \cite[Theorem 5.3]{NTMNA} that there exists $\Omega\subseteq 
\partial C$ convex such that $\omega(z;g)\subseteq \Omega$ for all $z\in \Sigma^\circ$. 
\end{proof}
Theorems \ref{thm:7.1} and \ref{thm:7.2}  confirm a conjecture by Karlsson and Nussbaum; see  \cite{Khb,NTMNA}, for an interesting class of Hilbert's metric nonexpansive mappings. The main point is that the arguments do not rely on the geometry of the domain.  They also imply Theorem \ref{thm:2}, as order-preserving homogeneous mappings $f\colon C^\circ\to C^\circ$ always satisfies the fixed point property on $C^\circ$ with respect to $d_C$ and each orbit of $g\colon \Sigma^\circ\to \Sigma^\circ$ has a compact closure in the norm topology, if the cone is finite dimensional. However, we do not know whether  there  exists an order-preserving homogenous mapping $f\colon C^\circ\to C^\circ$, where $C$ is a finite dimensional closed cone,  with a point $x\in C^\circ$ such that 
$\mathcal{O}(x;f)$ has an accumulation point in $\partial C$ and $\mathcal{O}(x;f)$ is unbounded in the norm topology. In fact, we conjecture that such a mapping cannot exist, but at present we can only prove it for polyhedral cones. 
\begin{theorem}\label{thm:7.4}
If $f\colon C^\circ\to C^\circ$ is an order-preserving homogeneous mapping on the interior of a polyhedral cone, then there does not exist a point $x\in C^\circ$ such that 
$\mathcal{O}(x;f)$ has an accumulation point in $\partial C$ and $\mathcal{O}(x;f)$ is unbounded in the norm topology.
\end{theorem}
Theorem \ref{thm:7.4} is a simple consequence of the following proposition. 
\begin{proposition}\label{prop:7.5} 
If $f\colon C^\circ\to C^\circ$ is an order-preserving homogenous mapping on the interior of a polyhedral cone $C$ in a finite dimensional vector space $V$ with $r_{C^\circ}(f)=1$, and $x\in C^\circ$ is such that $\mathcal{O}(x;f)$ is unbounded in the norm topology, then there exists $h_R\in\mathcal{H}_R$ such that 
\[
\lim_{k\to\infty} h_R(f^k(x)) =-\infty.
\]
\end{proposition} 
\begin{proof}
For simplicity we write $x_k:=f^k(x)$ and $z_k:=x_k/\|x_k\|$ for $k\geq 0$. As $\mathcal{O}(x;f)$ is unbounded in the norm topology there exists a subsequence 
$(x_{k_j})_j$ of $(x_k)_k$ such that 
\[
\|x_m\|<\|x_{k_j}\|\mbox{\quad for all }m<k_j.
\]
Note that we can take a further subsequence such that the $i_R(x_{k_j})$ converges to say $h_R\in\mathcal{H}_R$ and $z_{k_j}\to z\in C\setminus\{0\}$ as $j\to\infty$. We claim that $z\in\partial C$. Indeed, suppose, for the sake of contradiction, that $z\in C^\circ$. The mapping 
$g\colon y\mapsto \frac{f(y)}{\|f(y)\|}$ on $\Sigma^\circ:=\{y\in C^\circ\colon \|y\|=1\}$ is nonexpansive on $(\Sigma^\circ, \delta_C)$. Moreover, 
\[
g^{k_j}(z_0) =\frac{f^{k_j}(x_0)}{\|f^{k_j}(x_0)\|}= z_{k_j}\to z\in C^\circ,
\]
as $j\to\infty$. Thus, $\omega(z_0;g)\cap \Sigma^\circ$ is nonempty. It now follows from \cite[Corollary 3.2.5]{LNBook} that $g$ has a fixed point, say $u\in\Sigma^\circ$. The equality $u = g(u) =\frac{f(u)}{\|f(u)\|}$ implies that $\|f(u)\|=r_{C^\circ}(f)=1$. Thus, $f$ has a fixed point in $C^\circ$. As $f$ is nonexpansive under $d_C$ on $C^\circ$, it follows that all orbits of $f$ are bounded under $d_C$, and hence also bounded in the norm topology, as the topologies coincide. This contradicts our assumption; so,  $z\in\partial C$.  

Let $E$ be the extreme points of $\Sigma^*_{z_0}:=\{\phi\in C^*\colon \phi(z_0)=1\}$.
Note that $E$ is a finite set, as $C$ is polyhedral. Let $E_0:=\{\phi\in E\colon \phi(z) =0\}$ and $E_+:=E\setminus E_0$, which are both nonempty sets. 

Observe that for $m\geq 0$ fixed and $\phi \in E_0$ we have that 
\[
\log \frac{\phi(z_{k_j-m})}{\phi(z_{k_j})} \leq d_H(g^{k_j}(z_0),g^{k_j-m}(z_0))\leq d_H(g^{m}(z_0),g(z_0))<\infty.
\]
Thus, for $\phi\in E_0$, we have that $\phi(z_{k_j-m})\to 0$ as $j\to\infty$. As $z_{k_j}\to z$ and $\phi(z_{k_j})>0$ for all $j$ and $\phi \in E_+$, we know that there exists a constant $\gamma>0$ such that $\phi(z_{k_j})>\gamma$ for all $j$ and $\phi\in E_+$. Combining these two observations  gives that 
\[
\limsup_{j\to\infty} \RF_C(z_{k_i},z_{k_j-m}) = \limsup_{j\to\infty} \log\left (\frac{\phi_j(z_{k_j-m})}{\phi_j(z_{k_i})}\right )\leq -\log\gamma
\]
for some $\phi_j\in E_+$, as $\phi_j(z_{k_j-m})\leq 1$ and $\phi(z_{k_i})>\gamma$. 
Likewise, we have for all $j$ sufficiently large that 
\[
\RF_C(z_0,z_{k_j}) =\log\left (\frac{\phi_j(z_{k_j})}{\phi_j(z_0)}\right ) \geq \log \gamma,
\]
where $\phi_j\in  E_+$.

Now fix integers $m,i>0$ and consider 
\[
h_R(x_{k_i+m}) =\lim_{j\to\infty} \RF_C(x_{k_i+m},x_{k_j}) -\RF_C(x_0,x_{k_j}).
\]

As $f$ is order-preserving and homogeneous, it is nonexpansive with respect to $\RF_C$; see Lemma \ref{lem:5.2}(4). Therefore 
\begin{eqnarray*}
h_R(x_{k_i+m}) & \leq & \limsup_{j\to\infty} \RF_C(x_{k_i},x_{k_j-m}) - \RF_C(x_0,x_{k_j})\\
   & \leq &  \limsup_{j\to\infty} \RF_C(z_{k_i},z_{k_j-m}) - \RF_C(z_0,z_{k_j}) + \log \left ( \frac{ \|x_{k_j-m}\|\|x_0\|}{\|x_{k_i}\|\|x_{k_j}\|}\right )\\
   & \leq &  \limsup_{j\to\infty} \RF_C(z_{k_i},z_{k_j-m}) - \RF_C(z_0,z_{k_j})+\log \left(\frac{\|x_0\|}{\|x_{k_i}\|}\right )\\
   & \leq & -2\log\gamma +\log\|x_0\|-\log\|x_{k_i}\|.
\end{eqnarray*}
As $\|x_{k_i}\|\to\infty$, we see that $\lim_{k\to\infty} h_R(x_{k})=-\infty$. 
\end{proof}
Note that Example \ref{ex:par} shows that Proposition \ref{prop:7.5} does not hold for general cones. 

Let us now prove Theorem \ref{thm:7.4}. 
\begin{proof}[Proof of Theorem \ref{thm:7.4}]
We argue by contradiction. So suppose that $(f^{m_i}(x))_i$ is a norm bounded subsequence and $\mathcal{O}(x,f)$ is unbounded in the norm topology. Then there exists $\beta>0$ such that $f^{m_i}(x)\leq \beta x$ for all $i$. Before we can apply Proposition \ref{prop:7.5} we need to show that $r_{C^\circ}(f)=1$. Note that $\mathcal{O}(x;f)$ has a convergent subsequence $(f^{s_j}(x))_j$ with limit say $\eta\in C$. From Lemma \ref{cor:orbitzero} we know that $\eta\neq 0$, so that $r_{C^\circ}(f)=\lim_{j\to\infty} \|f^{s_j}(x)\|^{1/s_j}= 1$. 

By Proposition \ref{prop:7.5} there  also exists a subsequence  $(f^{k_j}(x))_j$ with $\|f^{k_j}(x)\|\to\infty$ such that $i_R(f^{k_j}(x))\to h_R\in\mathcal{H}_R$ such that $h_R(f^m(x))\to-\infty$ as $m\to\infty$. Note, however, that 
\begin{eqnarray*}
i_R(f^{k_j}(x))(f^{m_i}(x)) & = & \RF_C(f^{m_i}(x),f^{k_j}(x)) -\RF_C(x,f^{k_j}(x))\\
 & \geq & \RF_C(\beta x, f^{k_j}(x))-\RF_C(x,f^{k_j}(x)) \\
 & = & -\log\beta
\end{eqnarray*}
for all $i$ and $j$, which is absurd. 
\end{proof}

\begin{remark}
There exists an alternative proof of Theorem \ref{thm:7.4} that does not rely on horofunctions. We sketch the argument for the interested reader below. 

As $C$ is a polyhedral cone the order-preserving homogeneous mapping has a continuous order-preserving homogeneous extension to the whole of $C$; see \cite{BNS}. Moreover, it follows from \cite[Theorem 5.3.1 and Proposition 5.3.6]{LNBook} that $1=r_{C^\circ}(f)=\hat{r}_C(f)$, where $\hat{r}_C(f)$ is the Bonsall spectral radius, which is given by 
\[
\hat{r}_C(f):=\|f^k\|_C^{1/k}.
\] 
Now suppose that $x\in C^\circ$ and $\mathcal{O}(x;f)$ is unbounded in the norm topology. Then there exists a subsequence $(k_i)_i$ such that 
$\lim_{i\to\infty}\|f^{k_i}(x)\|=\infty$ and $\|f^j(x)\|<\|f^{k_i}(x)\|$ for  all $j<k_i$ and $i\geq 1$.
Assume that  we have selected a subsequence of $(k_i)_i$, which we also label by $k_i$, such that 
\[
\lim_{i\to\infty}\frac{f^{k_i-\nu}(x)}{\|f^{k_i-\nu}(x)\|} =:\eta_\nu\in C
\]
for all $\nu=0,\ldots,m$. Leave the subsequence unchanged for $i\leq m$, and for $i\geq m+1$ modify the subsequence so that 
\[
\lim_{i\to\infty}\frac{f^{k_i-(m+1)}(x)}{\|f^{k_i-(m+1)}(x)\|} =\eta_{m+1}
\]
for some $\eta_{m+1}\in C$. Continuing in this way, we obtain a subsequence $(k_i)_i$ such that 
\[
\lim_{i\to\infty}\frac{f^{k_i-\nu}(x)}{\|f^{k_i-\nu}(x)\|} =:\eta_\nu\mbox{\quad for all } 
\nu\geq 0.   
\]

Now note that, as $f^{k_m}(x)$ and $f^{k_m-m}(x)$ in $C^\circ$, there exist $0< a_m\leq b_m$ such that 
\[
a_m f^{k_m-m}(x)\leq_C f^{k_m}(x)\leq_C b_m f^{k_m-m}(x),
\]
and hence $a_m f^{k_j-m}(x)\leq_C f^{k_j}(x)\leq_C b_m f^{k_j-m}(x)$ for all $j\geq m$. This gives 
\[
a_m \frac{f^{k_j-m}(x)}{\|f^{k_j}(x)\|}\leq_C \frac{f^{k_j}(x)}{\|f^{k_j}(x)\|}\leq_C b_m \frac{f^{k_j-m}(x)}{\|f^{k_j}(x)\|}.
\]
As 
\[
\frac{\|f^{k_j-m}(x)\|}{\|f^{k_j}(x)\|}\leq 1\mbox{\quad and\quad }\frac{\|f^{k_j-m}(x)\|}{\|f^{k_j}(x)\|}\geq \frac{1}{\|f^m\|_C},
\]
we have that 
\[
\frac{a_m}{\|f^m\|_C} \frac{f^{k_j-m}(x)}{\|f^{k_j-m}(x)\|}\leq_C \frac{f^{k_j}(x)}{\|f^{k_j}(x)\|}\leq_C b_m \frac{f^{k_j-m}(x)}{\|f^{k_j-m}(x)\|}.
\]
Letting $j\to\infty$ gives
\[
\frac{a_m}{\|f^m\|_C}\eta_m\leq_C\eta_0\leq_Cb_m\eta_m\mbox{\quad for all }m\geq 1.
\]
Thus, $\eta_m\sim_C \eta_0$ for all $m\geq 1$, and hence $\eta_m\sim_C\eta_n$ for all $m,n\geq 1$. In a similar way it can be shown that $f(\eta_1)\sim_C \eta_{0}$. As $\eta_1\sim_C\eta_0$, it follows that for each $x\sim_C \eta_0$ we have that $f(x)\sim_C f(\eta_0)\sim_C\eta_0$, and hence $f(C_0)\subseteq C_0$, where $C_0:=\{x\in C\colon x\sim_C\eta_0\}$ is the part of $\eta_0$. By continuity of $f\colon C\to C$, we find that $f(\overline{C}_0)\subseteq \overline{C}_0$. 

It is known that $C_0$ is the relative interior of the closed cone $\overline{C}_0$; see \cite[Lemma 1.2.2]{LNBook}. We claim that $\hat{r}_{\overline{C}_0}(f_{\mid \overline{C}_0})=1$. Obviously $\hat{r}_{\overline{C}_0}(f_{\mid \overline{C}_0})\leq 1$, as $\hat{r}_C(f)=1$. Note that for all $m\geq 1$, we have that $\|\eta_m\|=1$,  $\eta_m\in C_0$, and 
\[
\|f^m(\eta_m)\|= \lim_{i\to\infty}\frac{\|f^m(f^{k_i-m}(x))\|}{\|f^{k_i-m}(x)\|} = 
 \lim_{i\to\infty}\frac{\|f^{k_i}(x)\|}{\|f^{k_i-m}(x)\|}\geq 1,
\]
so $\hat{r}_{\overline{C}_0}(f_{\mid \overline{C}_0})=\sup\{\|f^m(x)\|^{1/k}\colon x\in \overline{C}_0\mbox{ with }\|x\|\leq 1\}\geq 1$. 

It follows from \cite[Corollary 5.4.2]{LNBook} that there exists $v\in \overline{C}_0$ such that $f(v)=v$ and $\|v\|=1$. As $\eta_0\in C_0$, there exists $\beta>0$ such that $v\leq_C \beta\eta_0$. As $f^{k_i}(x)/\|f^{k_i}(x)\|\to \eta_0$ and $C$ is polyhedral, we know; see \cite[Lemma 5.1.4]{LNBook}, that for each $0<\lambda<1$ there exists $i_0\geq 1$ such that $\lambda \eta_0\leq_C f^{k_i}(x)/\|f^{k_i}(x)\|$ for all $i\geq i_0$. So, if we fix $0<\lambda<1$, and let $b:=\beta^{-1}$,  we get 
\[
b\lambda\|f^{k_i}(x)\|v = b\lambda\|f^{k_i}(x)\| f^{m}(v)\leq_C f^{k_i+m}(x)\mbox{\quad for all }i\geq i_0. 
\]
It follows that $\liminf_{m\to\infty} \|f^{k_i+m}(x)\|\geq b\lambda\kappa^{-1}\|f^{k_i}(x)\|$, where 
$\kappa>0$ is the normality constant of $C$, so that $\liminf_{n\to\infty} \|f^{n}(x)\| =\infty$. Thus, $\mathcal{O}(x;f)$ cannot have any accumulation points in $C$.
\end{remark}


\end{document}